\documentclass[10pt,reqno]{amsart}
\usepackage{amsmath,amsopn,amssymb,amsthm}
\usepackage[cp1251]{inputenc}
\usepackage[english]{babel}
\usepackage[pagebackref,breaklinks=true,colorlinks=true,linkcolor=blue,citecolor=blue,urlcolor=blue]{hyperref}
\usepackage{graphicx}%
\usepackage{color}

\newtheorem{theorem}{Theorem}[section]

\newtheorem{problem}[theorem]{Problem}

\newtheorem{definition}[theorem]{Definition}
\newtheorem{example}[theorem]{Example}

\newtheorem{lemma}[theorem]{Lemma}

\newtheorem{remark}[theorem]{Remark}

\renewenvironment{proof}[1][Proof]{\noindent\textbf{#1.} }{\ \rule{0.5em}{0.5em}}

\begin{document}
\title[Symmetric space, strongly isotropy irreducibility and equigeodesic properties]{Symmetric space, strongly isotropy irreducibility and equigeodesic properties}
\author{Ming Xu}
\address[Ming Xu]{
School of Mathematical Sciences,
Capital Normal University,
Beijing 100048,
P. R. China}
\email{mgmgmgxu@163.com}

\author{Ju Tan$^*$}
\address[Ju Tan]{School of Mathematics and Physics,
Anhui University of Technology, Maanshan, 243032, P.R. China}\email{tanju2007@163.com}

\date{}

\thanks{$^*$ Ju Tan is the corresponding author.}

\begin{abstract}
A smooth curve on a homogeneous manifold $G/H$ is called a Riemannian equigeo-desic if it is a homogeneous geodesic for any $G$-invariant Riemannian metric. The homogeneous manifold $G/H$ is called Riemannian equigeodesic, if for any $x\in G/H$ and any nonzero $y\in T_x(G/H)$, there exists a Riemannian equigeodesic $c(t)$ with $c(0)=x$ and $\dot{c}(0)=y$. These two notions can be naturally transferred to the Finsler setting, which provides the definitions for Finsler equigeodesic and Finsler equigeodesic space. We prove two classification theorems for Riemannian equigeodesic spaces and Finsler equigeodesic spaces respectively. Firstly, a homogeneous manifold $G/H$ with connected simply connected quasi compact $G$ and connected $H$ is Riemannian equigeodesic if and only if it can be decomposed as a product of Euclidean factors and compact strongly isotropy irreducible factors. Secondly, a homogeneous manifold $G/H$ with a compact semi simple $G$ is Finsler equigeodesic if and only if
it can be locally decomposed as a product, in which each factor is $Spin(7)/G_2$, $G_2/SU(3)$ or a symmetric space of compact type. These results imply that symmetric space and strongly isotropy irreducible space of compact type can be interpreted by equigeodesic properties. As an application, we classify the homogeneous manifold $G/H$ with a compact semi simple $G$, such that all $G$-invariant Finsler metrics on $G/H$ are Berwald. It suggests a new project in homogeneous Finsler geometry, to systematically study the homogeneous manifold $G/H$ on which all $G$-invariant Finsler metrics satisfy certain geometric property.


\end{abstract}

\maketitle

\section{Introduction}

A Riemannian manifold is called a {\it geodesic orbit} (or simply {\it g.o.}) space, if each maximally extended geodesic is homogeneous, i.e.,
it is the orbit of a one-parameter subgroup of isometries. This notion was introduced by
O. Kowalski and L. Vanhecke in 1991 \cite{KV1991}. Since then, it has been extensively studied in homogeneous Riemannian geometry and homogeneous pseudo-Riemannian geometry \cite{AA2007,AN2009,CWZ2022,DKN2004,Go1996,GN2018,Ni2013,NW2022} (more references can be found in \cite{BN2020}).
Recently, S. Deng and Z. Yan defined the g.o. property in homogeneous Finsler geometry \cite{YD2014}. See \cite{Xu2018,Xu2021,ZX2022} for some recent progress.

In this paper we are motivated by the following question:
\medskip

\noindent
{\bf Question}\quad
{\it Can we define an analog of the g.o. property with homogeneous geodesic replaced by equigeodesic?}
\medskip

Here an {\it equigeodesic} is a smooth curve on
a homogeneous manifold $G/H$, which is a homogeneous geodesic for any $G$-invariant Finsler metric, or any $G$-invariant Finsler metric in a preferred subclass, Riemannian, Randers, $(\alpha,\beta)$, etc. As we have different types of equigeodesics, we specify them as {\it Finsler equigeodesic}, {\it Riemannian equigeodesic}, {Rander equigeodesic}, {\it $(\alpha,\beta)$ equigeodesic}, etc.  Equigeodesic was firstly introduced by L. Cohen, L. Grama and C.J.C. Negreiros in 2010 \cite{CGN2010}. Until now, only Riemannian equigeodesic has been studied on some special homogeneous manifolds \cite{GN2011}\cite{St2020}\cite{WZ2013}. In this paper, we concern only Riemannian equigeodesic and Finsler equigeodesic. See \cite{TX2022} for some progress on other types of equigeodesics
.

Using the pattern of the g.o. property, we define {\it Riemannian} and {\it Finsler equigeodesic space} as following. A homogeneous manifold $G/H$ is called {\it Riemannian} or {\it Finsler equigeodesic}, if for any $x\in G/H$ and any nonzero $y\in T_x(G/H)$, there exists a equigeodesic $c(t)$ of the specified type satisfying $c(0)=x$ and $\dot{c}(0)=y$. When $G/H$ has an orthogonal reductive decomposition $\mathfrak{g}=\mathfrak{h}+\mathfrak{m}$ with respect to
a fixed $\mathrm{Ad}(G)$-invariant inner product $\langle\cdot,\cdot\rangle_{\mathrm{bi}}$ on $\mathfrak{g}$,
$G/H$ is Riemannian or Finsler equigeodesic if and only if each
$u\in\mathfrak{m}\backslash\{0\}$ is a Riemannian or Finsler equigeodesic vector, respectively.

Compared to the mixed nature of homogeneous geodesic and g.o. property, which is half algebraic and half geometric, equigeodesic and equigeodesic spaces are totally algebraic properties, and they are much stronger. So it looks more likely that equigeodesic and equigeodesic spaces can be explicitly described or completely classified, without two much calculation. This thought is justified by the main theorems of
this paper, which partially classified Riemannian and Finsler equigeodesic spaces.

For Riemannian equigeodesic space, we have
\medskip

{\bf Theorem A}\quad
{\it Let $G/H$ be a simply connected homogeneous manifold on which the connected simply connected quasi compact Lie group $G$ acts almost effectively. Then $G/H$ is a Riemannian equigeodesic if and only if it is a product of Euclidean factors and strongly isotropy irreducible factors.}
\medskip

Here a product decomposition $G/H=G_1/H_1\times \cdots\times G_m/H_m$ for a homogeneous manifold means
that $G=G_1\times\cdots\times G_m$, $H=H_1\times\cdots
H_n$ with $H_i=H\cap G_i$ for each $i$.

Theorem A is a reformulation of Theorem \ref{main-thm-1}. Firstly, it provides a new description for compact strongly isotropy irreducible spaces and their products (see \cite{Be1995} or Theorem 27 in \cite{BG2014}).
Secondly, it reduces the classification for some connected simply connected Riemannian equigeodesic spaces to that for compact strongly isotropy irreducible spaces \cite{Kr1975,Ma1961-1,Ma1961-2,Ma1966,Wo1968}. Finally, it implies that the Riemannian equigeodesic space property for $G/H$ depends not only on Lie algebras but also on Lie groups (see Remark \ref{remark}), so
our knowledge on Riemannian equigeodesic space which is not connected or not simply connected
is still quite limited.

For Finsler equigeodesic space, we have
\medskip

\noindent
{\bf Theorem B}\quad
{\it Let $G/H$ be a homogeneous manifold on which the compact semi simple Lie group $G$ acts almost effectively. Then $G/H$ is Finsler equigeodesic if and only if it can be locally decomposed as $G/H=G_1/H_1\times\cdots\times G_m/H_m$, in which each  $G_i/H_i$ is $Spin(7)/G_2$, $G_2/SU(3)$ or a symmetric space of compact type.}
\medskip

Here the local product decomposition for a homogeneous manifold $G/H$ means a product decomposition for the universal cover for a connected component $G_0/G_0\cap H$ of $G/H$,
or equivalently the following direct sum decompositions in the Lie algebra level,
\begin{eqnarray*}
\mathfrak{g}=
\mathfrak{g}_1\oplus\cdots\oplus\mathfrak{g}_{m},\quad
\mathfrak{h}=
\mathfrak{h}_1\oplus\cdots\oplus\mathfrak{h}_m
=(\mathfrak{h}\cap\mathfrak{g}_1)\oplus\cdots
(\mathfrak{h}\cap\mathfrak{g}_m).
\end{eqnarray*}

Theorem B reduces the classification for Finsler equigeodesic spaces of compact type to that for symmetric spaces \cite{He1962}. Compare to Theorem A, the proof of Theorem B is harder, but the classification result is much cleaner and more complete.

The strategy for proving Theorem B is the following. Firstly we prove two criterions for Finsler equigeodesic vector and Finsler equigeodesic space respectively (see Theorem \ref{main-thm-2} and Lemma \ref{lemma-3}). It turns out that Finsler equigeodesic space is a property which only depends on the Lie algebras. So, secondly, we can use Theorem A (notice that a Finsler equigeodesic space must be Riemannian equigeodesic) to locally decompose a Finsler equigeodesic space $G/H$ to a product of compact strongly isotropy irreducible factors. By Lemma \ref{lemma-4}, each factor is also Finsler equigeodesic. Then we need to apply the criterion to each $G/H$ in the classification list for compact non-symmetric strongly isotropy irreducible spaces. This would be a terribly long journey. Fortunately we find the short cut. Roughly speaking, if $G/H$ is Finsler equigeodesic, then $\dim\mathfrak{h}/\dim\mathfrak{g}$ can not be too small, and for the
orthogonal reductive decomposition $\mathfrak{g}=\mathfrak{h}+\mathfrak{m}$ with respect to
an $\mathrm{Ad}(G)$-invariant inner product $\langle\cdot,\cdot\rangle_{\mathrm{bi}}$ on $\mathfrak{g}$, each vector in $\mathfrak{m}$ has a relatively large centralizer in $\mathfrak{g}$ (see (1) and (3) in Lemma \ref{lemma-6}).
It is very easy to check that these numerical properties can not be simultaneously satisfied by most compact strongly isotropy irreducible $G/H$ (see Theorem \ref{main-thm-4}).

Theorem A and Theorem B provide new interpretations for strongly isotropy irreducibility and symmetric space. Though in this paper we have only studied equigeodesic space from the compact side, we guess those non-compact ones are also interesting and may share some similar phenomena.

Moreover, Theorem B starts a new project in homogeneous Finsler geometry, i.e., to classify the homogeneous manifold $G/H$ on which each $G$-invariant Finsler metric satisfies curtain geometric property.
For example, we may consider
\begin{problem}\label{problem-1}
Classify all homogeneous manifolds $G/H$ such that all $G$-invariant Finsler metrics on $G/H$ are
Berwald.
\end{problem}
We prove that
a homogeneous manifold $G/H$ with a compact $G$ is Finsler equigeodesic if and only if
it has a reductive decomposition for which each $G$-invariant Finsler metric is naturally reductive (see Lemma \ref{proposition-1}), and if and only if the property in Problem \ref{problem-1} is satisfied
(see Theorem \ref{thm-1}), so Theorem B has the following application, which helps us solve Problem \ref{problem-1} partially.\medskip

\noindent
{\bf Theorem C}\quad
{\it Let $G/H$ be a homogeneous manifold on which the compact semi simple Lie group $G$ acts almost effectively. Then it satisfies the condition that each $G$-invariant Finsler metric on $G/H$ is Berwald
if and only if  it can be locally decomposed as $G/H=G_1/H_1\times\cdots\times G_m/H_m$, in which each  $G_i/H_i$ is $Spin(7)/G_2$, $G_2/SU(3)$ or a symmetric space of compact type.}
\medskip

It seems promising that Theorem C might be generalized for noncompact $G$.

We may naturally generalize Question 5.12.25 in \cite{BN2020} and consider
\begin{problem}\label{problem-2}
Classify all homogeneous manifolds $G/H$ such that all $G$-invariant Finsler metrics on $G/H$ are
geodesic orbit (or have vanishing S-curvature with respect to a $G$-invariant measure).
\end{problem}
It seems that Problem \ref{problem-2} is a much harder problem.

This paper is organized as following. In Section 2, we summarize some basic knowledge on general and homogeneous Finsler geometry, and on flag manifold, which are necessary for our later discussion. In Section 3, we introduce Riemannian equigeodesic space and prove Theorem A (i.e., Theorem \ref{main-thm-1}). In Section 4, we introduce Finsler equigeodesic and Finsler equigeodesic space, and partially prove Theorem B. In Section 5, we discuss compact strongly
isotropy irreducible Finsler equigeodesic space, and finish
the proof of Theorem B. In Section 6, we prove Theorem C.

\section{Preliminaries}

\subsection{Minkowski norm, Finsler metric and geodesic}
A {\it Minkowski norm} on a real vector space $\mathbf{V}$, $\dim\mathbf{V}=n$, is a continuous function $F:\mathbf{V}\rightarrow \mathbb{R}_{\geq0}$ satisfying
\begin{enumerate}
\item the positiveness and smoothness, i.e., $F|_\mathbf{V}\backslash0$ is a positive smooth function;
\item the positive 1-homogeneity, i.e., $F(\lambda y)=\lambda F(y)$ for every $\lambda\geq0$;
\item the convexity, i.e., for any $y\in\mathbf{V}\backslash\{0\}$,
$\langle u,v\rangle^F_y=\tfrac12 \tfrac{\partial^2}{\partial s\partial t}|_{s=t=0}F^2(y+su+tv)$ defines an inner product on $\mathfrak{m}$.
\end{enumerate}

A {\it Finsler metric} on a smooth manifold $M$ is a continuous function $F:TM\rightarrow \mathbb{R}_{\geq0}$ such that
$F|_{TM\backslash0}$ is smooth, and for each $x\in M$,
$F(x,\cdot)$ is a Minkowski norm. We also call
$(M,F)$ a {\it Finsler manifold} or a {\it Finsler space}
\cite{BCS2000}.

If each $F(x,\cdot)$ is an Euclidean norm, i.e., $F(x,y)=\langle y,y\rangle^{1/2}$ for some inner product $\langle\cdot,\cdot\rangle$ on $T_xM$, we say $F$ is a {\it Riemannian metric}. A Minkowski norm (or a Finsler metric) $F$ is Euclidean (or Riemannian respectively) if its Cartan tensor $$C^F_y(u,v,w)=\tfrac14\tfrac{\partial^3}{\partial r\partial s\partial t}|_{r=s=t=0}F^2(y+ru+sv+tw)$$ vanishes everywhere.

Geodesic on a Finsler manifold $(M,F)$ can be similarly defined as in Riemannian geometry, which is a smooth curve satisfying the locally minimizing principle for the arch length functional. Practically, we always assume that a geodesic has positive constant speed, i.e., $F(c(t),\dot{c}(t))\equiv\mathrm{const}>0$. Then a smooth curve $c(t)$ is a geodesic if and only if its lifting $(c(t),\dot{c}(t))$ in $TM\backslash0$ is an integral curve of the geodesic spray $$\mathrm{G}=y^i\partial_{x^i}-2\mathrm{G}^i\partial_{y^i},
\mbox{ in which } \mathrm{G}^i=\tfrac14g^{il}([F^2]_{x^ky^l}y^k-[F^2]_{x^l}).$$
Locally a geodesic $c(t)$ is a solution of ODE system
$$\ddot{c}^i(t)+2\mathrm{G}(c(t),\dot{c}(t))=0,\quad\forall i.$$

See \cite{BCS2000,Sh2001} for more details.

\subsection{Homogeneous Finsler space and invariant Finsler metric}

A Finsler manifold $(M,F)$ is {\it homogeneous} if
its isometry group $I(M,F)$ acts transitively on $M$. Since $I(M,F)$ is a Lie transformation group, we can present $M$ as
a homogeneous manifold $M=G/H$, for any Lie subgroup $G\subset I(M,F)$ which acts transitively on $M$, and the homogeneous metric $F$ is also called a {\it $G$-invariant metric} \cite{De2012}. In this definition $G$ must act  effectively on $(G/H,F)$. However, in most occasions, the {\it almost effectiveness}, i.e., $\mathfrak{h}$ does not contain a nonzero ideal of $\mathfrak{g}$, is enough. So we choose another way to introduce homogeneous Finsler geometry and its basic algebraic setups.

In this paper, we use $G/H$ to denote a homogeneous manifold, denote $\mathfrak{g}=\mathrm{Lie}(G)$, $\mathfrak{h}=\mathrm{Lie}(H)$, and always assume the following:
\begin{enumerate}
\item $G$ acts almost effectively on $G/H$;
\item$H$ is compactly imbedded in $G$, i.e., $\mathrm{Ad}_\mathfrak{g}(H)$ has a compact closure in $\mathrm{Aut}\mathfrak{g}$.
\end{enumerate}

The assumption (2) guarantees a {\it reductive decomposition} for $G/H$, i.e., an $\mathrm{Ad}(H)$-invariant decomposition
$
\mathfrak{g}=\mathfrak{h}+\mathfrak{m}
$. With respect to the given reductive decomposition, we denote $\mathrm{pr}_\mathfrak{h}:\mathfrak{g}
\rightarrow\mathfrak{h}$ and $\mathrm{pr}_\mathfrak{m}:\mathfrak{g}
\rightarrow\mathfrak{m}$
the corresponding linear projections, and denote $u_\mathfrak{h}=\mathrm{pr}_\mathfrak{h}(u)$,
$u_\mathfrak{m}=\mathrm{pr}_\mathfrak{m}(u)$ for any $u\in\mathfrak{g}$.
The subspace $\mathfrak{m}$ can be naturally viewed as the tangent space $T_o (G/H)$ at the origin $o=eH$ with the
$\mathrm{Ad}(H)$-action on $\mathfrak{m}$ identified with the isotropy $H$-action on $T_o(G/H)$.

The assumption (2) also guarantees the existences of $G$-invariant Riemannian and Finsler metrics on $G/H$.
By the homogeneity, a $G$-invariant Riemannian metric on $G/H$ is one-to-one determined by an $\mathrm{Ad}(H)$-invariant inner product $\langle\cdot,\cdot\rangle$ on $\mathfrak{m}=T_e(G/H)$.
Similarly, a $G$-invariant Finsler metric $F$ on $G/H$ can be uniquely determined by $F(o,\cdot)$, which can be any arbitrary  $\mathrm{Ad}(H)$-invariant Minkowski norm on $\mathfrak{m}$.
For simplicity, we still use the same $F$ to denote it.
Recall that the Hessian of $\tfrac12F^2$ provides a family of inner products on $\mathfrak{m}$, i.e., $\langle\cdot,\cdot\rangle^F_y$, $y\in\mathfrak{m}\backslash0$.

In homogeneous Finsler geometry, the following fundamental result is well known.

\begin{lemma} \label{lemma-0}
For a homogeneous Finsler manifold $(G/H,F)$
with a reductive decomposition $\mathfrak{g}=\mathfrak{h}+\mathfrak{m}$, we have
\begin{eqnarray*}
\langle [v,w_1],w_2\rangle_y^F+\langle w_1,[v,w_2]\rangle_y^F+2C^F_y(w_1,w_2,[v,y])=0,\quad
\forall v\in\mathfrak{h},w_1,w_2\in\mathfrak{m}, y\in
\mathfrak{m}\backslash\{0\}.
\end{eqnarray*}
\end{lemma}

In later discussion, we sometimes replace the assumption (2) with the following even stronger assumption: $\mathrm{Ad}_\mathfrak{g}(G)$ is a compact.
Then we can find, and then fix
an $\mathrm{Ad}(G)$-invariant inner product $\langle\cdot,\cdot\rangle_{\mathrm{bi}}$ on $\mathfrak{g}$.
For simplicity, we also call the pair  $(G,\langle\cdot,\cdot\rangle_{\mathrm{bi}})$ a {\it quasi compact} Lie group. In this situation, the reductive decomposition $\mathfrak{g}=\mathfrak{h}+\mathfrak{m}$ is chosen to be the $\langle\cdot,\cdot\rangle_{\mathrm{bi}}$-orthogonal one, which is simply called
an {\it orthogonal reductive decomposition}, and
any $G$-invariant Riemannian metric, determined by the inner product $\langle,\rangle$ on $\mathfrak{m}$, one-to-one determines the
metric operator $\Lambda:\mathfrak{m}\rightarrow\mathfrak{m}$ by
$$\langle u,v\rangle=\langle u,\Lambda(v)\rangle_{\mathrm{bi}},\quad\forall u,v\in\mathfrak{m},$$
 which exhausts all $\mathrm{Ad}(H)$-invariant $\langle\cdot,\cdot\rangle_{\mathrm{bi}}$-positive definite
linear endomorphisms on $\mathfrak{m}$.

\subsection{Homogeneous geodesic}

A geodesic $c(t)$ on a Finsler manifold $(M,F)$ is called {\it homogeneous}, if it is the orbit of a one-parameter subgroup of isometries \cite{YD2014}. For a homogeneous Finsler space $(G/H,F)$, as the isometry subgroup $G$ has been specified,
a homogeneous geodesic is then required to have the form $c(t)=\exp tu\cdot x$ for some $u\in\mathfrak{g}$ and $x\in M$. In particular, when $c(t)=\exp tu\cdot o$ is
a homogeneous geodesic, we call the vector $u\in\mathfrak{g}$
a {\it geodesic vector} for $(G/H,F)$. The follow lemma in \cite{La2007} is a well known criterion for geodesic vector.

\begin{lemma} \label{lemma-1}
For a homogeneous Finsler space $(G/H,F)$
with a reductive decomposition $\mathfrak{g}=\mathfrak{h}+\mathfrak{m}$,  $u\in\mathfrak{g}$ is a geodesic vector if and only if  $u\notin\mathfrak{h}$ and it satisfies
\begin{equation}\label{001}
\langle u_\mathfrak{m},
[\mathfrak{m},u]_\mathfrak{m}\rangle^F_{u_\mathfrak{m}}
=0.
\end{equation}
\end{lemma}
When $F$ is a $G$-invariant Riemannian metric, determined by the $\mathrm{Ad}(H)$-invariant inner product $\langle\cdot,\cdot\rangle$ on $\mathfrak{m}$, this criterion for geodesic vector is still valid \cite{KV1991}, and we may simplify (\ref{001}) as
\begin{equation*}
\langle u_\mathfrak{m},[\mathfrak{m},u]_\mathfrak{m}
\rangle=0.
\end{equation*}

\subsection{Homogeneous Berwald space and naturally reductive Finsler space}
\label{section-2-4}
A Finsler manifold $(M,F)$ is called {\it Berwald} if its geodesic spray $\mathrm{G}=y^i\partial_{x^i}-2\mathrm{G}^i\partial_{y^i}$ is {\it affine}, i.e., all the coefficients
$\mathrm{G}^i=\mathrm{G}^i(x,y)$ are quadratic for its $y$-entry \cite{Sh2001}.

For a homogeneous Finsler manifold $(G/H,F)$ with a reductive decomposition $\mathfrak{g}=\mathfrak{h}+\mathfrak{m}$, its Berwald property can be described using
the {\it spray vector field} $\eta$ introduced by L. Huang \cite{Hu2015-1,Hu2015-2}, i.e.
a smooth map $\eta:\mathfrak{m}\backslash\{0\}\rightarrow\mathfrak{m}$ satisfying
$$\langle\eta(y),u\rangle_y=\langle y,[u,y]_\mathfrak{m}
\rangle_y,\quad\forall u\in\mathfrak{m}.$$

\begin{lemma}\label{lemma-9}
A homogeneous Finsler manifold $(G/H,F)$ with a reductive decomposition $\mathfrak{g}=\mathfrak{h}+\mathfrak{m}$ is Berwald if and only if its spray vector field
$\eta:\mathfrak{m}\backslash\{0\}\rightarrow\mathfrak{m}$ is a quadratic map.
\end{lemma}

\begin{proof} It has been pointed out in Section 5 of \cite{Xu2022} that the geodesic spray $\mathrm{G}$ of $(G/H,F)$ can be decomposed as $\mathrm{G}=\mathrm{G}_0-\mathrm{H}$, where
$\mathrm{G}_0$ is the spray structure for the Nomizu connection on $G/H$, with respect to the given reductive decomposition, and $\mathrm{H}$ is a $G$-invariant vector field on $T(G/H)\backslash0$
which is tangent to each $T_x(G/H)$ and
$\mathrm{H}|_{T_o(G/H)\backslash\{0\}=\mathfrak{m}\backslash\{0\}}=\eta$. Since the Nomizu connection is a linear connection on $G/H$ \cite{KN1969}, its corresponding spray structure $\mathrm{G}_0$ is affine. So $\mathrm{G}$ is affine if and only if $\mathrm{H}$ is quadratic when restricted to each slit tangent space $T_x(G/H)\backslash\{0\}$, and by the $G$-invariancy of $\mathrm{H}$, if and only if $\eta=\mathrm{H}|_{T_o(G/H)\backslash\{0\}}$ is quadratic.
\end{proof}

Naturally reductive Finsler manifolds are a special class of homogeneous Berwald metrics \cite{La2007,DH2010}. A homogeneous Finsler manifold $(G/H,F)$ is called {\it naturally reductive} with respect to a given
reductive decomposition $\mathfrak{g}=\mathfrak{h}+\mathfrak{m}$ if each curve $c(t)=\exp tu \cdot o$ with $u\in\mathfrak{m}\backslash\{0\}$ is a geodesic, or equivalently it has a vanishing
spray vector field $\eta:\mathfrak{m}\backslash\{0\}\rightarrow\mathfrak{m}$ \cite{DH2010}.
\subsection{Classification for flag manifolds}
\label{section-2-5}
Let $G$ be a compact connected semi simple Lie group, then for any vector $u\in\mathfrak{g}$, the adjoint orbit $\mathrm{Ad}(G)u
\subset\mathfrak{g}$ is called a {\it flag manifold}.

A flag manifold $\mathrm{Ad}(G)u$ can be presented as a homogeneous manifold $G/C_G(u)$. To determine the Lie algebra $\mathrm{Lie}(C_G(u))=\mathfrak{c}_\mathfrak{g}(u)
=\mathfrak{c}(\mathfrak{c}_\mathfrak{g}(u))\oplus
[\mathfrak{c}_\mathfrak{g}(u),\mathfrak{c}_\mathfrak{g}(u)]$,
we only need to determine its central summand
$\mathfrak{c}(\mathfrak{c}_\mathfrak{g}(u))$ and its
semi simple summand $[\mathfrak{c}_\mathfrak{g}(u),\mathfrak{c}_\mathfrak{g}(u)]$
as following \cite{Al1997}.

Let $\mathfrak{t}$ be a Cartan subalgebra, for which we have
the root system $\Delta_\mathfrak{g}$ of $\mathfrak{g}$, and
a prime root system $\{\alpha_1,\cdots,\alpha_n\}$, where $n=\dim\mathfrak{t}$ is the rank. We usually use an $\mathrm{Ad}(G)$-invariant inner product $\langle\cdot,\cdot\rangle_{\mathrm{bi}}$ on $\mathfrak{g}$ to identify $\mathfrak{t}$ with $\mathfrak{t}^*$,
then the roots are viewed as vectors in $\mathfrak{t}$.
Obviously $\mathfrak{c}_\mathfrak{g}(u)$ contains $\mathfrak{t}$, so $\mathfrak{c}_\mathfrak{g}(u)$ is a {\it regular subalgebra} of $\mathfrak{g}$, i.e., each root or root space of $\mathfrak{c}_\mathfrak{g}(u)$ is
also a root or root space of $\mathfrak{g}$ respectively.

By suitable $G$-conjugation or Weyl group action, we may assume that $u$ satisfies $\langle\alpha_i,u\rangle_{\mathrm{bi}}\geq0$ for each $i$.
Suppose we have $\langle\alpha_i,u\rangle_{\mathrm{bi}}=0$ for $1\leq i\leq k$ and $\langle\alpha_i,u\rangle_{\mathrm{bi}}>0$ for $k<i\leq n$. Then the central summand
$\mathfrak{c}(\mathfrak{c}_\mathfrak{g}(u))$ in $\mathfrak{c}_\mathfrak{g}(u)$ is $(n-k)$-dimensional, linearly spanned by $\alpha_i$ with $k<i\leq n$. For the semi simple summand $[\mathfrak{c}_\mathfrak{g}(u),
\mathfrak{c}_\mathfrak{g}(u)]$,
$\{\alpha_1,\cdots,\alpha_k\}$ provides a prime root system, i.e., each root of $\mathfrak{g}$ belongs to $[\mathfrak{c}_\mathfrak{g}(u),
\mathfrak{c}_\mathfrak{g}(u)]$
if and only if it is a linear combination of $\alpha_i$ with $1\leq i\leq k$. To get the Dynkin diagram for $[\mathfrak{c}_\mathfrak{g}(u),
\mathfrak{c}_\mathfrak{g}(u)]$, we may start with that for $\mathfrak{g}$ and then delete
all the dots for $\alpha_i$ with $k<i\leq n$ and all the edges connecting to to these dots.

\section{Riemannian equigeodesic and Riemannian equigeodesic space}
\subsection{Riemannian equigeodesic and Riemannian equigeodesic vector}
A smooth curve $c(t)$ on a homogeneous manifold $G/H$ is called a {\it Riemannian equigeodesic}, if it is a homogeneous geodesic for any $G$-invariant Riemannian metric on $G/H$. A vector $u\in\mathfrak{g}$ is called
a {\it Riemannian equigeodesic vector} if $c(t)=\exp tu\cdot o$ is a Riemannian equigeodesic.

Now suppose $(G,\langle\cdot,\cdot\rangle_{\mathrm{bi}})$ is quasi compact and $\mathfrak{g}=\mathfrak{h}+\mathfrak{m}$ is the corresponding orthogonal reductive decomposition for $G/H$.
Using Lemma \ref{lemma-1}, the following criterion for Riemannian equigeodesic vector can be easily deduced.

\begin{lemma}\label{lemma-2}
Let $G/H$ be a homogeneous manifold with a quasi compact $(G,\langle\cdot,\cdot\rangle_{\mathrm{bi}})$ and the corresponding orthogonal reductive decomposition $\mathfrak{g}=\mathfrak{h}+\mathfrak{m}$. Then $u\in\mathfrak{g}$ is an equigeodesic vector for $G/H$
if and only if $u\notin\mathfrak{h}$ and $[\Lambda(u_\mathfrak{m}),u]_\mathfrak{m}=0$ for every metric operator $\Lambda$.
\end{lemma}

Since we can choose $\Lambda=\mathrm{id}$, Lemma
\ref{lemma-2} provides $[u_\mathfrak{m},u_\mathfrak{h}]=
[u_\mathfrak{m},u]_\mathfrak{m}=0$ when $u$ is a Riemannian equigedesic vector, i.e., $u$ and $u_\mathfrak{m}$ generates the same Riemannian equigeodesic $c(t)=\exp tu \cdot o=\exp tu_\mathfrak{m}\exp tu_\mathfrak{h}\cdot o=\exp tu_\mathfrak{m}\cdot o$.

To summarize, discussion for Riemannian equigeodesics on $G/H$ with a quasi compact $(G,\langle\cdot,\cdot\rangle_{\mathrm{bi}})$ and the corresponding orthogonal reductive decomposition $\mathfrak{g}=\mathfrak{h}+\mathfrak{m}$ can be reduced to that for Riemannian equigeodesic vectors in $\mathfrak{m}\backslash\{0\}$, i.e., $c(t)=\exp tv \cdot x$ is
a Riemannian equigeodesic passing $x=g\cdot o$ if and only if $(\mathrm{Ad}(g^{-1})v)_\mathfrak{m}$ is a Riemannian equigeodesic vector.
\label{Section-3-1}
\subsection{Riemannian equigeodesic space}
Now we define a Riemannian equigeodesic space.

\begin{definition} \label{definition-1}
We call a homogeneous manifold $G/H$ Riemannian equigeodesic or a Riemannian equigeodesic space
if for each $x\in G/H$ and each nonzero $y\in T_x(G/H)$, there exists a Riemannian equigeodesic $c(t)$ with $c(0)=x$ and $\dot{c}(0)=y$.
\end{definition}


When $G$ is quasi compact, we can use Lemma \ref{lemma-2} and the observation
in the last paragraph of Section \ref{Section-3-1} to give
the following equivalent description for Riemannian equigeodesic space (the proof is easy and skipped).

\begin{lemma}\label{lemma-7}
Let $G/H$ be a homogeneous manifold with quasi compact
$(G,\langle\cdot,\cdot\rangle_{\mathrm{bi}})$ and the corresponding orthogonal reductive decomposition $\mathfrak{g}=\mathfrak{h}+\mathfrak{m}$. Then $G/H$
is {Riemannian equigeodesic} if and only if each nonzero $u\in\mathfrak{m}$ is a Riemannian equigeodesic vector, i.e.,
$[\Lambda(u),u]_\mathfrak{m}=0$ for each $u\in\mathfrak{m}\backslash\{0\}$ and every metric operator $\Lambda$.
\end{lemma}

For example, an isotropy irreducible $G/H$ with a compact $G$ is Riemannian equigeodesic. A connected Abelian Lie group $G=G/\{e\}$ is Riemannian equigeodesic. More generally, we have the following example.
%
\begin{example}\label{example-1}
Suppose that
$G/H=(G_0\times \cdots\times G_k)/(H_0\times \cdots\times H_k)=G_0/H_0\times G_1/H_1\times\cdots\times G_k/H_k$,
in which $G_0$ is a connected Abelian Lie group, $H_0=\{e\}$, and for each $i>0$,
$G_i$ and $H_i$ are compact and $G_i/H_i$ is isotropy irreducible. Applying Schur Lemma, any metric operator for $G/H$ has the form $\Lambda=\Lambda_0\oplus c_1\mathrm{id}|_{\mathfrak{m}_1}\oplus\cdots\oplus c_k\mathrm{id}|_{\mathfrak{m}_k}$, in which $\Lambda_0$ is an endomorphism on $\mathfrak{g}_0$.
Using Lemma \ref{lemma-7}, it is easy to see that $G/H$ is a Riemannian equigeodesic space.
\end{example}

The following theorem indicates that Riemannian equigeodesic space is such a strong condition that Example \ref{example-1} becomes a typical model.

\begin{theorem}\label{main-thm-1}
Let $G/H$ be a homogeneous manifold on which $G$ acts almost effectively. Suppose that $G$ is connected, simply connected and quasi compact, and $H$ is a connected. Then
$G/H$ is Riemannian equigeodesic if and only if we have the decompositions $G=G_1\times \cdots\times G_m$ and $H=H_1\times\cdots\times H_m$, such that each $G_i/H_i$
is one of the following
\begin{enumerate}
\item a real line, i.e., $G_i=\mathbb{R}$ and $H_i=\{0\}$;
\item a strongly isotropy irreducible $G_i/H_i$ with compact connected semi simple $G_i$ and connected $H_i$.
\end{enumerate}
\end{theorem}

Recall that a homogeneous manifold $G/H$ is {\it isotropy irreducible} if the isotropic $H$-action is irreducible, and it is {\it strongly isotropy irreducible} if  the isotropy action is irreducible when restricted to $H_0$. Compact strongly isotropy irreducible spaces and compact isotropy irreducible spaces have been classified in \cite{Kr1975,Ma1961-1,Ma1961-2,Ma1966,Wo1968} and \cite{WZ1991} respectively. There exists many examples of isotropy irreducible $G/H$ which is not strongly isotropy irreducible.

\begin{proof}
Firstly we assume that $G/H$ is Riemannian equigeodesic and prove the decompositions.

Let $\langle\cdot,\cdot\rangle_{\mathrm{bi}}$ be an $\mathrm{Ad}(G)$-invariant inner product on $\mathfrak{g}$ and $\mathfrak{g}=\mathfrak{h}+\mathfrak{m}$
the corresponding orthogonal reductive decomposition. We  further $\mathrm{Ad}(H)$-equivariantly decompose $\mathfrak{m}$ as $\mathfrak{m}=\mathfrak{m}_1+\cdots+\mathfrak{m}_m$, such that the each $\mathfrak{m}_i$ is irreducible.

{\bf Claim 1}: $[\mathfrak{m}_i,\mathfrak{m}_j]\subset\mathfrak{h}$ when $i\neq j$.

For any $u_i\in\mathfrak{m}_i$ and $u_j\in\mathfrak{m}_j$,
we apply Lemma \ref{lemma-7} to $u=u_i+u_j$ and  $\Lambda=\mathrm{id}|_{\mathfrak{m}_i}\oplus
2\mathrm{id}|_{\sum_{l\neq i}\mathfrak{m}_l}$, and get $[u_i,u_j]_\mathfrak{m}=[\Lambda(u),u]_\mathfrak{m}=0$,
which proves Claim 1.

{\bf Claim 2}: $[\mathfrak{m}_i,\mathfrak{m}_i]
\subset\mathfrak{h}+\mathfrak{m}_i$
for each $i$.

By the $\mathrm{Ad}(G)$-invariancy of $\langle\cdot,\cdot\rangle_{\mathrm{bi}}$,
\begin{eqnarray*}
\langle [\mathfrak{m}_i,\mathfrak{m}_i],\mathfrak{m}_j \rangle_{\mathrm{bi}}
\subset\langle [\mathfrak{m}_i,\mathfrak{m}_j],\mathfrak{m}_i \rangle_{\mathrm{bi}}
\subset\langle\mathfrak{h},\mathfrak{m}_i\rangle_{\mathrm{bi}}
=0,\quad\forall j\neq i.
\end{eqnarray*}
So we have $[\mathfrak{m}_i,\mathfrak{m}_i]\subset\mathfrak{h}+\mathfrak{m}_i$,
which proves Claim 2.

{\bf Claim 3}: $[\mathfrak{m}_i,\mathfrak{m}_j]=0$ when $i\neq j$.

By the almost effectiveness of the $G$-action on $G/H$, we only need to prove that $[\mathfrak{m}_i,\mathfrak{m}_j]$ with $i\neq j$ is an ideal of $\mathfrak{g}$ contained in $\mathfrak{h}$.

It is obvious that $[\mathfrak{m}_i,\mathfrak{m}_j]$ is an ideal of $\mathfrak{h}$, i.e., $[[\mathfrak{m}_i,\mathfrak{m}_j],\mathfrak{h}]\subset
[\mathfrak{m}_i,\mathfrak{m}_j]$.
When $i\neq j\neq k\neq i$, we have
$[[\mathfrak{m}_i,\mathfrak{m}_j],\mathfrak{m}_k]=0$
because by Claim 1, $[[\mathfrak{m}_i,\mathfrak{m}_j],\mathfrak{m}_k]\subset
[\mathfrak{h},\mathfrak{m}_k]\subset\mathfrak{m}_{k}$ on one hand,
and
$[[\mathfrak{m}_i,\mathfrak{m}_j],\mathfrak{m}_k]
\subset[[\mathfrak{m}_i,\mathfrak{m}_k],\mathfrak{m}_j]+
[\mathfrak{m}_i,[\mathfrak{m}_j,\mathfrak{m}_k]]
\subset \mathfrak{m}_j+\mathfrak{m}_i$ on the other hand.
The subspace $[[\mathfrak{m}_i,\mathfrak{m}_j],\mathfrak{m}_i]\subset[\mathfrak{h},\mathfrak{m}_i]\subset\mathfrak{m}_i$
vanishes because
\begin{eqnarray*}
& &\langle[[\mathfrak{m}_i,\mathfrak{m}_j],\mathfrak{m}_i],
\mathfrak{m}_i\rangle_{\mathrm{bi}}
=\langle[[\mathfrak{m}_i,\mathfrak{m}_j],[\mathfrak{m}_i,
\mathfrak{m}_i]\rangle_{\mathrm{bi}}
=\langle[\mathfrak{m}_j,[\mathfrak{m}_i,[\mathfrak{m}_i,
\mathfrak{m}_i]]\rangle_{\mathrm{bi}}\\
&= &
\subset\langle\mathfrak{m}_j,[\mathfrak{m}_i,\mathfrak{h}+
\mathfrak{m}_i]\rangle_{\mathrm{bi}}
\subset\langle\mathfrak{m}_j,\mathfrak{h}+
\mathfrak{m}_i\rangle_{\mathrm{bi}}=0,
\end{eqnarray*}
where we have used Claim 2. For the same reason,
$[\mathfrak{m}_i,\mathfrak{m}_j],\mathfrak{m}_j]$ also vanishes. Summarizing above argument, we see $[\mathfrak{m}_i,\mathfrak{m}_j]\subset\mathfrak{h}$ with $i\neq j$ is an ideal of $\mathfrak{g}$. Then Claim 3 is proved.

{\bf Claim 4}: we have Lie algebra direct sum decompositions $\mathfrak{g}=\mathfrak{g}_1\oplus
\cdots\oplus\mathfrak{g}_m$ and $\mathfrak{h}=\mathfrak{h}_1\oplus\cdots\oplus\mathfrak{h}_m$, in which $\mathfrak{g}_i=[\mathfrak{m}_i,\mathfrak{m}_i]+
\mathfrak{m}_i$ and $\mathfrak{h}_i=[\mathfrak{m}_i,\mathfrak{m}_i]_\mathfrak{h}$.

By Claim 2 and Claim 3, each $\mathfrak{g}_i\subset\mathfrak{g}$ is an ideal of $\mathfrak{g}$ with $\mathfrak{g}_i\cap\mathfrak{m}=\mathfrak{m}_i$.  So we have a Lie algebra direct sum decomposition
$\mathfrak{g}=\mathfrak{g}_1\oplus\cdots\oplus
\mathfrak{g}_m
\oplus\mathfrak{g}'$, in which the ideal $\mathfrak{g}'$ is the $\langle\cdot,\cdot\rangle_{\mathrm{bi}}$-orthogonal complement of $\mathfrak{g}_1\oplus\cdots\oplus\mathfrak{g}_m$.
Obviously, $\mathfrak{m}\subset\mathfrak{g}_1\oplus\cdots
\oplus\mathfrak{g}_m$, so the ideal $\mathfrak{g}'$ of $\mathfrak{g}$ is contained in $\mathfrak{h}$, which must vanish by the almost effectiveness. So we get the decomposition for $\mathfrak{g}$.
Meanwhile Claim 2 implies that each $\mathfrak{g}_i$ is compatible with the orthogonal reductive decomposition, i.e. $\mathfrak{g}_i=\mathfrak{h}_i+\mathfrak{m}_i$, in which $\mathfrak{h}_i=\mathfrak{g}_i\cap\mathfrak{h}=
[\mathfrak{m}_i,\mathfrak{m}_i]_\mathfrak{h}$
and $\mathfrak{m}_i=\mathfrak{g}_i\cap\mathfrak{g}$.
The decomposition for $\mathfrak{h}$ follows immediately.
Now Claim 4 is proved.

Finally, we consider the corresponding decompositions for $G$, $H$ and $G/H$.

Since $G$ is connected and simply connected, Claim 4 provides the Lie group product decomposition $G=G_1\times\cdots\times G_m$, in which each $G_i$ is a connected simply connected quasi compact Lie subgroup generated by $\mathfrak{g}_i$.
In each $G_i$, we have a connected Lie subgroup $H_i$ with $\mathrm{Lie}(H_i)=\mathfrak{h}_i$. Then $H= H_1\times\cdots\times H_n$ because both sides are connected Lie groups generated by the same Lie subalgebra. Obviously $H_i=H\cap G_i$. By the closeness and connectedness of $H$, each $H_i$ is a closed connected subgroup of $G_i$. The almost effectiveness of the $G$-action on $G/H$ implies the almost effectiveness of the $G_i$-action on $G_i/H_i$. In the decomposition $G/H=G_1/H_1\times \cdots\times G_m/H_m$, each $G_i/H_i$ has an strongly irreducible isotropy representation. The classification in \cite{Wo1968} indicates that $G_i$ must be compact and semi simple, unless
$G_i=\mathbb{R}$ and $H_i=\{0\}$.

To summarize, above argument proves Theorem \ref{main-thm-1} in one direction. The other direction is obvious by the discussion for Example \ref{example-1}.
\end{proof}

This proof is self contained. It can be simplified using some results in Section 5 of \cite{Ni2017}. Another possible proof is to verify that each $G$-invariant Riemannian metric on $G/H$ is normal and then use an analog of the main theorem in \cite{Be1995}.

\begin{remark}\label{remark}
Theorem \ref{main-thm-1} only classifies some simply connected Riemannian equigeodesic spaces.
The universal cover of a Riemannian equigeodesic space may not be Riemannian equigeodesic any more. For example, $G/H=SU(3)/T^2\mathbb{Z}_3$ is isotropy irreducible, so it is Riemannian equigeodesic. Its universal cover $SU(3)/T^2$ is not Riemannian equigeodesic by Theorem \ref{main-thm-1}, because it is not strongly isotropy irreducible. Using the classification work in \cite{WZ1991}, many other similar examples can be found.
\end{remark}
\section{Finsler equigeodesic and Finsler equigeodesic space}
\subsection{Finsler equigeodesic and Finsler equigeodesic vector}
The definitions for {\it Finsler equigeodesic} and {\it Finsler equigeodesic vector} were proposed in \cite{TX2022}.

\begin{definition}
A smooth curve on $G/H$ is called a {\it Finsler equigeodesic} if it is a homogeneous geodesic for any $G$-invariant Finsler metric
on $G/H$. A vector $u\in\mathfrak{g}$ is called a {\it Finsler equigeodesic vector} if it generates an equigeodesic $c(t)=\exp tu\cdot o$ on $G/H$.
\end{definition}

Equivalently speaking, $u$ is a Finsler equigeodesic vector if and only if it is a geodesic vector for any $G$-invariant Finsler metric on $G/H$.
Obviously
any Finsler equigeodesic is also
a Riemannian equigeodesic.  So Finsler equigeodesic
is a stronger algebraic property than Riemannian equigeodesic.
For Finsler equigeodesic vector, the observation is similar.
\subsection{A Criterion and some examples}
Let $G/H$ be a homogeneous manifold with compact
$(G,\langle\cdot,\cdot\rangle_{\mathrm{bi}})$, and the corresponding orthogonal irreducible decomposition $\mathfrak{g}=\mathfrak{h}+\mathfrak{m}$. By the observation in Section \ref{Section-3-1}, the discussion for Finsler equigeodesics can be reduced to that for Finsler equigeodesic vectors in $\mathfrak{m}\backslash\{0\}$.

For any vector $u\in\mathfrak{m}\backslash\{0\}$, we denote $H_u=\{g\in H| \mathrm{Ad}(g) u=u\}$, and $\mathbf{V}_u$ the $\langle\cdot,\cdot\rangle_{\mathrm{bi}}$-orthogonal complement of $[\mathfrak{h},u]$ in $\mathfrak{m}$.
Then the
$\mathrm{Ad}(H_u)$-action preserves $\mathbf{V}_u$.
So we can further decompose $\mathbf{V}_u$ as
$\mathbf{V}_u=\mathbf{V}_{u,0}+\mathbf{V}_{u,1}$, in which
$\mathbf{V}_{u,0}$ is the fixed point set of the $\mathrm{Ad}(H_u)$-action and $\mathbf{V}_{u,1}$ is the
$\langle
\cdot,\cdot\rangle_{\mathrm{bi}}$-orthogonal complement of
$\mathbf{V}_{u,0}$ in $\mathbf{V}_u$.

With above settings, the criterion for $u$ to be a Finsler equigeodesic vector is the following.

\begin{theorem} \label{main-thm-2}
Keep all above assumptions and notations, then
any vector $u\in\mathfrak{m}\backslash\{0\}$ is a Finsler equigeodesic vector if and only if  it satisfies
\begin{equation}\label{003} [u,\mathfrak{m}]_\mathfrak{m}\subset[u,\mathfrak{h}]+
\mathbf{V}_{u,1}.
\end{equation}
\end{theorem}

\begin{proof} We first assume $u\in\mathfrak{m}\backslash\{0\}$
satisfies (\ref{003}) and prove it is a Finsler equigeodesic.

Let $F$ be any
$G$-invariant Finsler metric on $G/H$.
Denote $\mathbf{W}_{u,F}=\{w\in\mathfrak{m}| \langle u,w\rangle^F_u=0\}$ the subspace of all directions in which the derivative of $F$ vanishes at $u$.

By Lemma \ref{lemma-0} and the property $C^F_u(u,\cdot,\cdot)\equiv0$ for the Cartan tensor, we immediately get
$\langle u,[\mathfrak{h},u]\rangle^F_u=0$, i.e., $[\mathfrak{h},u]\subset \mathbf{W}_u^F$.

Let $w$ be any vector in $\mathbf{V}_{u,1}$.
Since the Minkowski norm $F=F(o,\cdot)$ on $\mathfrak{m}$
is $\mathrm{Ad}(H)$-invariant, for any $g\in H_u$, we have
\begin{equation}\label{004}
\langle \mathrm{Ad}(g)w,u
\rangle^F_{u}=
\langle \mathrm{Ad}(g)w,\mathrm{Ad}(g)u
\rangle^F_{\mathrm{Ad}(g)u}=\langle w,u\rangle_u^F.
\end{equation}
Since $G$ is compact, $H$ and $H_u$ are also compact. So we can
integrate (\ref{004}) over $H_u$ and get
$$\langle\int_{g\in H_u}\mathrm{Ad}(g)w d\mathrm{vol}_{H_u},u\rangle_u^F=\mathrm{vol}(H_u)\cdot \langle w,u\rangle^F_u,$$
where $d\mathrm{vol}$ is a bi-invariant measure on $H_u$ and $\mathrm{vol}(H_u)=\int_{H_u}d\mathrm{vol}_{H_u}\in(0,+\infty)$.
Since $\int_{g\in H_u}\mathrm{Ad}(g)w d\mathrm{vol}_{H_u}\in \mathbf{V}_{u,1}$ and it is fixed all
$\mathrm{Ad}(H_u)$-actions, i.e., it is also contained in $\mathbf{V}_{u,0}$, we get $\int_{g\in H_u}\mathrm{Ad}(g)w d\mathrm{vol}_{H_u}=0$ and then $\langle w,u\rangle_u^F=0$.

To summarize, we combine above arguments and (\ref{003}), then we see
$$[\mathfrak{m},u]_\mathfrak{m}\subset
[\mathfrak{h},u]+\mathbf{V}_{u,1}\subset \mathbf{W}^F_u.$$
It implies
$\langle u,[\mathfrak{m},u]_\mathfrak{m}\rangle_u^F=0$, i.e.,
$u$ is a geodesic vector for $(G/H,F)$. Since $F$ is chosen arbitrarily, $u$ is a Finsler equigeodesic vector. This ends the proof of Theorem \ref{main-thm-2} in one direction.

We then prove the other direction, i.e., a Finsler equigeodesic $u$ must satisfy (\ref{003}).

Assume conversely that the Finsler equigeodesic vector $u\in\mathfrak{m}\backslash\{0\}$ does not satisfy (\ref{003}). Lemma \ref{lemma-1} implies
$[\mathfrak{m},u]_\mathfrak{m}\subset\bigcap_F \mathbf{W}_u^F$, where the intersection is taken for all $G$-invariant Finsler metrics on $G/H$. So we only need to prove $[u,\mathfrak{h}]+\mathbf{V}_{u,1}=\bigcap_{F}
\mathbf{W}_u^F$.

In above argument, we have already proved $[u,\mathfrak{h}]+\mathbf{V}_{u,1}\subset\bigcap_{F}
\mathbf{W}_u^F$.
To prove the inverse inclusion, we consider any nonzero vector $w\in\mathbf{V}_{u,0}$ and look for a $G$-invariant Finsler metric on $G/H$ with $w\notin\mathbf{W}^F_u$. The construction of $F$ is the following.

Without loss of generality, we may assume $\langle u,u\rangle_{\mathrm{bi}}=1$. By the slice theorem \cite{MY1957}, the compact $\mathrm{Ad}(H)$-action on the unit sphere $S=\{u| \langle u,u\rangle_{\mathrm{bi}}=1,u\in\mathfrak{m}\}\subset\mathfrak{m}$ provides an orbit type stratification $S=S_1\coprod\cdots\coprod S_N$. For each $i$, $S_i$ is an imbedded submanifold in $S$ with
a smooth fiber bundle structure, in which each fiber is
an $\mathrm{Ad}(H)$-orbit of the same type. We may assume  $\mathrm{Ad}(H)u\subset S_1$, then we can find a sufficiently small $\mathrm{Ad}(H)$-invariant open neighborhood $\mathcal{U}$ of $\mathrm{Ad}(H)u$ in $S_1$, such that the quotient map $\pi:\mathcal{U}\rightarrow\mathcal{U}/H\cong \mathbb{R}^k$ is a smooth fiber bundle, the fibers are the
$\mathrm{Ad}(H)$-orbits parametrized as $\mathcal{O}_{x_1,\cdots,x_k}$ for $(x_1,\cdots,x_k)\in\mathbb{R}^k$, with $\mathcal{O}_{0,\cdots,0}=\mathrm{Ad}(H)u$.
Since $\mathbf{V}_{u,0}$ is tangent to $\mathcal{U}$ and
$\mathbf{V}_{u,0}\cap T_u(\mathrm{Ad}(H)u)=0$, we may adjust the parameter space $\mathbb{R}^k$ such that
$\pi_*(w)$ coincides with $\tfrac{\partial}{\partial{x_1}}$ at the origin. We can find a smooth real function $\varphi$ on $\mathbb{R}^k$ with $\tfrac{\partial}{\partial x_1}\varphi(0,\cdots,0)\neq0$ and a sufficiently small compact support.
The function $\varphi\circ\pi$ is viewed as a $\mathrm{Ad}(H)$-invariant smooth function $S_1$ which only takes zero values outside $\mathcal{U}$.

Nextly, we thicken $\mathcal{U}$ to a $\mathrm{Ad}(H)$-invariant neighborhood $\mathcal{U}'$ of $\mathrm{Ad}(H)u$ in $S$. The function $\varphi\circ\pi$ can
be further extended to a compactly supported smooth function $\psi$ on $\mathcal{U}'$. By the averaging process for the $\mathrm{Ad}(H)$-action, the $\mathrm{Ad}(H)$-invariancy of $\psi$ can be achieved.

Finally, with above preparations, we are ready to construct the $G$-invariant Finsler metric.
For any sufficiently small $\epsilon>0$,
\begin{equation}\label{051}
F(y)=\langle y,y\rangle_{\mathrm{bi}}^{1/2}\cdot
(1+\epsilon\cdot\psi(\tfrac{y}{\langle y,y\rangle_{\mathrm{bi}}^{1/2}}))
\end{equation}
induces an $\mathrm{Ad}(H)$-invariant Minkowski norm on $\mathfrak{m}$. The derivative of $F$ at $u$ does not vanish
in the direction of $w$. So
for the $G$-invariant Finsler metric $F$ determined by this Minkowski norm,  we have $w\notin \mathbf{W}^F_u$.

This ends the proof of Theorem \ref{main-thm-2} in the other direction.
\end{proof}

As immediate applications for Theorem \ref{main-thm-2}, we have the following examples.

\begin{example}\label{example-2}
Let $G/H$ be a symmetric space of compact type with the Cartan decomposition $\mathfrak{g}=\mathfrak{h}+\mathfrak{m}$. Then any vector
$u\in\mathfrak{m}\backslash\{0\}$ is a Finsler equigeodesic vector.
\end{example}

Here we call the homogeneous manifold $G/H$ a {\it symmetric space of compact type} if
$G$ is compact semi simple, and $G/H$ has a {\it Cartan decomposition}, i.e., a reductive decomposition $\mathfrak{g}=\mathfrak{h}+\mathfrak{m}$
satisfying $[\mathfrak{m},\mathfrak{m}]\subset\mathfrak{h}$.
Notice that the Cartan decomposition is orthogonal with respect
to the Killing form $B_\mathfrak{g}$ of $\mathfrak{g}$ and we may take $\langle\cdot,\cdot\rangle_{\mathrm{bi}}=
-B_\mathfrak{g}(\cdot,\cdot)$. So the Cartan decomposition
is the corresponding orthogonal reductive decomposition. By Theorem \ref{main-thm-2}, each $u\in\mathfrak{m}\backslash\{0\}$ is a Finsler equigeodesic vector for $G/H$ because $[\mathfrak{m},u]_\mathfrak{m}=0\subset$  in this situation. This example motivates us to study Finsler equigeodesic spaces
(see Section \ref{section-4-3}).

\begin{example}
Let $G/H$ be a homogeneous manifold with compact semi simple $G$ and $\mathrm{rk}G=\mathrm{rk}H$. Then the set
of Finsler equigeodesic vectors is nonempty.
\end{example}
More precisely, Lemma 5.3 in \cite{XDHH2017} provides
the following Finsler equigeodesic vectors.
Let $\mathfrak{t}$ be a Cartan subalgebra of $\mathfrak{g}$ which is contained in $\mathfrak{h}$. Then the reductive decomposition $\mathfrak{g}=\mathfrak{h}+\mathfrak{m}$ for $G/H$ is unique, and we have the following root plane decompositions:
\begin{eqnarray*}
\mathfrak{g}=\mathfrak{t}+\sum_{\alpha\in\Delta_\mathfrak{g}}
\mathfrak{g}_{\pm\alpha},\quad
\mathfrak{h}=\mathfrak{t}+\sum_{\alpha\in\Delta_\mathfrak{h}}
\mathfrak{g}_{\pm\alpha},\quad
\mathfrak{m}=\sum_{\alpha\in \Delta_{\mathfrak{g}}\backslash\Delta_{\mathfrak{h}}}
\mathfrak{g}_{\pm\alpha},
\end{eqnarray*}
where $\Delta_\mathfrak{g}$ and $\Delta_\mathfrak{h}$ are the root systems of $\mathfrak{g}$ and $\mathfrak{h}$ respectively.
Then for any $\alpha\in\Delta_\mathfrak{g}\backslash\Delta_{\mathfrak{h}}$, any vector $u\in\mathfrak{g}_{\pm\alpha}\backslash\{0\}\subset\mathfrak{m}$ is a Finsler equigeodesic vector. This fact can be explained by Theorem \ref{main-thm-2}, because in this occasion we have
$[\mathfrak{g}_{\pm\alpha},\mathfrak{m}]_\mathfrak{m}
\subset
\sum_{\alpha\neq\beta\in\Delta_{\mathfrak{g}}\backslash\Delta_{\mathfrak{h}}
}\mathfrak{g}_{\pm\beta}\subset
\mathbf{V}_{u,1}$.

\begin{example} For a compact Lie group $G$,  $u\in\mathfrak{g}\backslash\{0\}$ is a Finsler equigeodesic vector for $G=G/\{e\}$ if and only if $u\in\mathfrak{c}(\mathfrak{g})$. So in this case, the Riemannian equigeodesics (Riemannian equigeodesic vectors) and the Finsler equigeodesics (Finsler equigeodesic vectors respectively) are the same.
\end{example}

\subsection{Finsler equigeodesic space and a criterion}
\label{section-4-3}
As an analog of Definition \ref{definition-1},
we define  {\it Finsler equigeodesic space} as the following.

\begin{definition}\label{definition-2}
We call a homogeneous manifold $G/H$ Finsler equigeodesic or a Finsler equigeodesic space
if for each $x\in G/H$ and each nonzero $y\in T_x(G/H)$, there exists a Finsler equigeodesic $c(t)$ with $c(0)=x$ and $\dot{c}(0)=y$.
\end{definition}

Let $G/H$ be a homogeneous manifold with compact $(G,
\langle\cdot,\cdot\rangle_{\mathrm{bi}})$, and the corresponding orthogonal reductive decomposition $\mathfrak{g}=\mathfrak{h}+\mathfrak{m}$. Then by the observation in Section \ref{Section-3-1},
Finsler equigeodesic space can be equivalently described as the following.

\begin{lemma}\label{lemma-8}
A homogeneous manifold $G/H$ with a compact $(G,
\langle\cdot,\cdot\rangle_{\mathrm{bi}})$ and the corresponding orthogonal reductive decomposition $\mathfrak{g}=\mathfrak{h}+\mathfrak{m}$
is {Finsler equigeodesic} if and only if each $u\in\mathfrak{m}\backslash\{0\}$ is a Finsler equigeodesic vector.
\end{lemma}

The Finsler equigeodesic property for $G/H$ with a compact $G$ can be described by natural reductiveness as following.

\begin{lemma}\label{proposition-1}
Let $G/H$ be a homogeneous manifold with a compact $(G,
\langle\cdot,\cdot\rangle_{\mathrm{bi}})$. Then we have the following:
\begin{enumerate}
\item if it has a reductive decomposition $\mathfrak{g}=\mathfrak{h}+\mathfrak{m}'$, with respect to which all $G$-invariant Finsler metrics on $G/H$ are naturally reductive, then $G/H$ is a Finsler equigeodesic space;
\item if $G/H$ is a Finsler equigeodesic space, then any $G$-invariant Finsler metric on $G/H$ is naturally reductive with respect to the orthogonal reductive decomposition $\mathfrak{g}=\mathfrak{h}+\mathfrak{m}$.
\end{enumerate}
\end{lemma}

\begin{proof} (1)
By the description for Finsler natural reductiveness in \cite{DH2010} or Section
\ref{section-2-4}, each nonzero vector in $\mathfrak{m}'$ is a Finsler equigeodesic vector for $G/H$. The projection $\mathrm{pr}_\mathfrak{m}:\mathfrak{g}\rightarrow\mathfrak{m}$ is a linear isomorphism when
restricted to $\mathfrak{m}'$. By the discussion after Lemma \ref{lemma-2}, each nonzero vector in
$\mathfrak{m}$ is Finsler equigeodesic. So $G/H$ is Finsler equigeodesic by Lemma \ref{lemma-8}.

(2) Let $F$ be any $G$-invariant Finsler metric on $G/H$. Lemma \ref{lemma-8} indicates that each nonzero vector $u\in\mathfrak{m}$ generates a geodesic $c(t)=\exp tu\cdot o$ on $(G/H,F)$. By the description in \cite{DH2010} or Section
\ref{section-2-4}, $(G/H,F)$
is naturally reductive with respect to the orthogonal reductive decomposition $\mathfrak{g}=\mathfrak{h}+\mathfrak{m}$.
\end{proof}

By Theorem \ref{main-thm-2} and Lemma \ref{lemma-8},  we have the following criterion for Finsler equigeodesic spaces.

\begin{lemma}\label{lemma-3}
A homogeneous manifold $G/H$ with a compact $(G,
\langle\cdot,\cdot\rangle_{\mathrm{bi}})$ and the corresponding orthogonal reductive decomposition $\mathfrak{g}=\mathfrak{h}+\mathfrak{m}$ is Finsler equigeodesic if and only if
\begin{equation}\label{033}
\mbox{there exists a conic open dense subset }\mathcal{U}\mbox{ of }\mathfrak{m}\backslash\{0\} \mbox{ satisfying }
[\mathfrak{m},u]_\mathfrak{m}
\subset[\mathfrak{h},u], \forall u\in\mathcal{U}.
\end{equation}
\end{lemma}
\begin{proof}
Assume $G/H$ is Finsler equigeodesic, then we can take $\mathcal{U}$ to be the union of all principal $\mathrm{Ad}(H)$-orbits in $\mathfrak{m}$. By the slice theorem for a linear group action \cite{HH1970,MY1957}, $\mathcal{U}$
is a conic open dense subset of $\mathfrak{m}$.
 For any $u\in\mathcal{U}$, $\mathrm{Ad}(H)$ acts trivially on
the $\langle\cdot,\cdot\rangle_{\mathrm{bi}}$-orthogonal complement of $T_u(\mathrm{Ad}(H)u)=[\mathfrak{h},u]$
in $\mathfrak{m}$, i.e., $\mathbf{V}_{u,1}=0$. So Theorem
\ref{main-thm-2} provides $[\mathfrak{m},u]_\mathfrak{m}
\subset[\mathfrak{h},u]$.

Assume $[\mathfrak{m},u]_\mathfrak{m}
\subset[\mathfrak{h},u]$ for each $u$ in a conic dense open subset $\mathcal{U}\subset\mathfrak{m}\backslash\{0\}$. Then for any $G$-invariant Finsler metric $F$ on $G/H$,  we have (\ref{001}), i.e., $\langle u,[v,u]_\mathfrak{m}\rangle_u^F=0$, $\forall v\in\mathfrak{m}$. By the continuity, (\ref{001}) is satisfied for all $u\in\mathfrak{m}\backslash\{0\}$. Since $F$ is arbitrarily chosen, we see $G/H$ is Finsler equigeodesic.
\end{proof}

The criterion Lemma \ref{lemma-3} reveals
an interesting phenomenon for the Finsler equigeodesic property, i.e., it is only relevant to Lie algebras. So we have the following immediate consequences.

\begin{lemma}\label{lemma-4}
A homogeneous manifold $G/H=(G_1\times G_2)/(H_1\times H_2)=G_1/H_1\times G_2/H_2$ with a compact $G$ is Finsler equigeodesic if and only if each $G_i/H_i$ is
Finsler equigeodesic.
\end{lemma}
\begin{proof} The proof repeatedly uses Lemma \ref{lemma-3}.
Assume $G/H$ is Finsler equigeodesic, which provides
$\mathcal{U}\subset\mathfrak{m}\backslash\{0\}$.
Denote $\mathrm{pr}_i:\mathfrak{m}\rightarrow\mathfrak{m}_i$
the linear projection according to
$\mathfrak{m}=\mathfrak{m}_1+\mathfrak{m}_2$. Then we can take
$\mathcal{U}_i=\mathrm{pr}_i(\mathcal{U}\backslash(\mathfrak{m}_1
\cup\mathfrak{m}_2))$ for $G_i/H_i$.
Assume each $G_i/H_i$ is Finsler equigeodesic,  providing $\mathcal{U}_i\subset\mathfrak{m}_i\backslash\{0\}$. Then
we can choose
$\mathcal{U}=\mathcal{U}_1\times \mathcal{U}_2$  for $G/H$.
\end{proof}

\begin{lemma}\label{lemma-5}
Let $G/H$ be a homogeneous manifold with a compact semi simple $G$. Then the following statements are equivalent:
\begin{enumerate}
\item $G/H$ is Finsler equigeodesic;
\item $G_0/(G_0\cap H)$ is Finsler equigeodesic, in which $G_0$ is the identity component of $G$;
\item $G/H_0$ is Finsler equigeodesic, in which $H_0$ is the identity component of $H$;
\item $\widetilde{G}_0/\widetilde{H}_0$ is Finsler equigeodesic, in which $\widetilde{G}_0$ is the universal cover of $G_0$ and $\widetilde{H}_0$ is the connected subgroup in $\widetilde{G}_0$ covering $H_0$.
\end{enumerate}
\end{lemma}
\begin{proof} Since $G$ is compact, $G_0$ and $\widetilde{G}_0$ are also compact, i.e., Lemma \ref{lemma-3} is applicable for each in (1)-(4).
The homogeneous manifolds in (1)-(4) of Lemma \ref{lemma-5} share the same $\mathfrak{g}$ and $\mathfrak{h}$, so they also share the same statement (\ref{033}) in Lemma \ref{lemma-3}.
\end{proof}

\subsection{Proof of Theorem B}

Now we prove Theorem B by the following steps, in which the details for Step 3 is postponed to Section \ref{section-5}.

{\bf Step 1}. We can use Lemma \ref{lemma-5} to replace $G/H$ by $\widetilde{G}_0/\widetilde{H}_0$, i.e., we may assume $G$ is compact, connected and simply connected, and $H$ is connected.

{\bf Step 2}. The Finsler equigeodesic space $G/H$ is also Riemannian equigeodesic, so Theorem \ref{main-thm-1} can be applied to decompose $G/H$ as
$G/H=G_1/H_1\times \cdots\times G_m/H_m$, in which each
$G_i/H_i$ is strongly isotropy irreducible. By Lemma \ref{lemma-4}, $G/H$ can be replaced by each $G_i/H_i$, i.e.,
we may further assume $G/H$ is a compact strongly isotropy irreducible space on which the compact semi simple $G$ acts almost effectively.

{\bf Step 3}. In Section \ref{section-5}, we will classify  strongly isotropy irreducible compact Finsler equigeodesic space $G/H$ in the Lie algebra level (see Theorem \ref{main-thm-4}), i.e.,
locally, $G/H$ must be one of the following
\begin{equation}\label{006}
Spin(7)/G_2,\mbox{ or } G_2/SU(3),\mbox{ or a symmetric space of compact type.}
\end{equation}

To summarize, above steps provide a local decomposition
$G/H=G_1/H_1\times\cdots\times G_m/H_m$, in which each
$G_i/H_i$ satisfies (\ref{006}).

{\bf Step 4}. We prove if a homogeneous manifold $G/H$
with a compact semi simple $G$ has a local decomposition
$G/H=G_1/H_1\times\cdots\times G_m/H_m$ in which each
$G_i/H_i$ satisfies (\ref{006}) is Finsler equigeodesic. Obviously
 $S^7=Spin(7)/G_2$ and $S^6=G_2/SU(3)$ are Finsler equigeodesic because
the invariant Finsler metric on each of them is unique up to a scalar, i.e., a Riemannian metric with positive constant curvature. In Example \ref{example-2}, we see a symmetric space of compact type is Finsler equigeodesic.
Using Lemma \ref{lemma-4} and Lemma \ref{lemma-5}, we see that $G/H$ is Finsler equigeodesic.

This ends the proof of Theorem B.


\section{Strongly isotropy irreducible compact Finsler equigeodesic space}
\label{section-5}

The goal of this section is to prove the following classification result.
\begin{theorem}\label{main-thm-4}
Let $G/H$ be a Finsler equigeodesic space on which the compact connected semi simple $G$ acts
almost effectively with a strongly irreducible isotropy representation. Then the pair $(\mathfrak{g},\mathfrak{h})$
is $(so(7),G_2)$, $(G_2,su(3))$, or a symmetric pair.
\end{theorem}

To prove this theorem, we need two preparations. Firstly, the numerical properties of a Finsler equigeodesic space in the following lemma are crucial for later case-by-case discussion.

\begin{lemma}\label{lemma-6}
Let $G/H$ be a homogeneous manifold with a compact connected semi simple $(G,
\langle\cdot,\cdot\rangle_{\mathrm{bi}})$ and the corresponding orthogonal reductive decomposition $\mathfrak{g}=\mathfrak{h}+\mathfrak{m}$. Suppose  that $G/H$ is Finsler equigeodesic, then we have the following:
\begin{enumerate}
\item $ \dim \mathfrak{c}_{\mathfrak{g}}(u)\geq\dim\mathfrak{g}-
    2\dim\mathfrak{h}$ for any
each $u\in\mathfrak{m}$;
\item $\dim \mathrm{Ad}(G)u+\dim
\mathfrak{c}(\mathfrak{c}_{\mathfrak{g}}(u))\geq
\dim\mathfrak{m}$,
in which $u\in\mathfrak{m}$ satisfies $$\dim\mathfrak{c}(\mathfrak{c}_{\mathfrak{g}}(u))=
\max_{v\in\mathfrak{m}}
\mathfrak{c}(\mathfrak{c}_\mathfrak{g}(v));$$
\item $ 2\dim \mathfrak{h}+ \mathrm{rk}\mathfrak{g}
   >\dim\mathfrak{m}$.
    \end{enumerate}
\end{lemma}
\begin{proof} (1)
Lemma \ref{lemma-3} provides a conic dense open subset
$\mathcal{U}\subset\mathfrak{m}\backslash\{0\}$. Let $u$ be any vector in $\mathcal{U}$, then $[\mathfrak{m},u]_{\mathfrak{m}}\subset[\mathfrak{h},u]$.

On one hand, we claim that the image of $\mathrm{pr}_\mathfrak{m}|_{\mathfrak{c}_{\mathfrak{g}}(u)}:
\mathfrak{c}_{\mathfrak{g}}(u)\rightarrow\mathfrak{m}$ is the kernel of the linear map $l(\cdot)=[\cdot,u]_{\mathfrak{h}}|_{\mathfrak{m}}:\mathfrak{m}
\rightarrow\mathfrak{h}$.

For any $w'\in\mathfrak{h}$ and $w\in\mathfrak{m}$, we have
$$[w'+w,u]=[w',u]+[w,u]_\mathfrak{m}+[w,u]_\mathfrak{h},$$
where the first two summands in the right side is contained in $\mathfrak{m}$ and the third is contained in $\mathfrak{h}$.
So $[w'+w,u]=0$ implies $[w,u]_\mathfrak{h}=0$, i.e., $w\in\ker l$.

For any $w\in\ker l$, the property of $u$ provides a vector $w'\in\mathfrak{h}$, such that $[w,u]_\mathfrak{m}=-[w',u]$. Then $v=w'+w$ satisfies
$$[v,u]=[w',u]+[w,u]_\mathfrak{h}+[w,u]_\mathfrak{m}
=-[w,u]_\mathfrak{m}+[w,u]_\mathfrak{m}=0,$$
i.e., there exists a vector $v\in\mathfrak{c}_\mathfrak{g}(u)$ with $v_\mathfrak{m} =w$. This ends the proof of our claim.

On the other hand, we see the obvious fact that the kernel of $\mathrm{pr}_\mathfrak{m}|_{\mathfrak{c}_{\mathfrak{g}}(u)}$
is $\mathfrak{c}_\mathfrak{h}(u)$.

Summarize above two observations, we get
\begin{equation}\label{035}
\dim\mathfrak{c}_\mathfrak{g}(u)=\dim \mathfrak{c}_\mathfrak{h}(u)+\dim \ker l,
\end{equation}
which implies
$$\dim\mathfrak{c}_{\mathfrak{g}}(u)\geq
\dim \ker l\geq\dim\mathfrak{m}-\dim\mathfrak{h}
=\dim\mathfrak{g}-2\dim\mathfrak{h}.$$
This proves (1) for $u\in\mathcal{U}$.

Notice that $\dim \mathfrak{c}_{\mathfrak{g}}(u)$ depends semi continuously on $u\in\mathfrak{m}$, i.e.,
$$\dim \mathfrak{c}_{\mathfrak{g}}(\lim_{n\rightarrow\infty} u_n)\geq
\overline{\lim}_{n\rightarrow\infty}
\dim\mathfrak{c}_{\mathfrak{g}}
(u_n).$$
This semi continuity implies that (1) is valid on $\overline{\mathcal{U}}=\mathfrak{m}$.

(2) Suppose $\max_{v\in\mathfrak{m}}
\dim\mathfrak{c}(\mathfrak{c}_\mathfrak{g}(v))$ is achieved at $u\in\mathfrak{m}$. Denote by $\mathfrak{t}$ any Cartan subalgebra containing $u$, and by $\mathfrak{g}=S_1\coprod\cdots\coprod S_N$ the orbit type stratification for the $\mathrm{Ad}(G)$-action.
We assume $S_1$ contains $u$. Then the quotient map
$\pi:S_1\rightarrow S_1/G$ is a smooth fiber bundle satisfying:
\begin{enumerate}
\item
each fiber is an $\mathrm{Ad}(G)$-orbit with same orbit type as $\mathrm{Ad}(G)u$;
\item
Locally around $u$, $S_1\cap\mathfrak{t}$ is section for
this fiber bundle and it is a linear subspace of dimension
$\dim\mathfrak{c}(\mathfrak{c}_\mathfrak{g}(u))$, which is the intersection of some Weyl walls.
\end{enumerate}

On the other hand, by the classification of flag manifolds (see \cite{Al1997} or Section \ref{section-2-5}), for any $w\in\mathfrak{g}\backslash S_1$
which is sufficiently close to $u$, we have $\dim\mathfrak{c}
(\mathfrak{c}_{\mathfrak{g}}(w))>\dim\mathfrak{c}(
\mathfrak{c}_\mathfrak{g}(u))$. So
the assumption
$\dim\mathfrak{c}(\mathfrak{c}_\mathfrak{g}(u))=
\max_{v\in\mathfrak{m}}\dim\mathfrak{c}
(\mathfrak{c}_\mathfrak{g}(v))$ implies that there exists a neighborhood of $u$ in $\mathfrak{m}$ which is contained in $S_1$.

To summarize,
we have
$$\dim S_1=\dim\mathrm{Ad}(G)u+
\dim\mathfrak{c}(\mathfrak{c}_\mathfrak{g}(u))\geq\dim\mathfrak{m},$$
which proves (2).


(3) For the vector $u\in\mathfrak{m}$ provided by (2), we have
\begin{equation}\label{010}
\dim\mathfrak{c}(\mathfrak{c}_\mathfrak{g}(u))\leq\dim\mathfrak{t}
=\mathrm{rk}\mathfrak{g}
\end{equation}
which is obvious, and
\begin{equation}\label{009}
\dim(\mathrm{Ad}(g)u)=\dim\mathfrak{g}-\dim
\mathfrak{c}_{\mathfrak{g}}(u)\leq 2\dim\mathfrak{h}
\end{equation}
by (1) of Lemma \ref{lemma-6}.
Input (\ref{010}) and (\ref{009}) into the equality in (2) of Lemma \ref{lemma-6}, we get
\begin{equation}\label{008}
2\dim\mathfrak{h}+\mathrm{rk}\mathfrak{g}\geq\dim\mathfrak{m}.
\end{equation}
%
%
So, to prove (3), we only need to verify that the equality in (\ref{008}) can not happen.

Assume conversely
$
2\dim \mathfrak{h}+ \mathrm{rk}\mathfrak{g}
   =\dim\mathfrak{m}
$.
Then $u$ is a regular vector, i.e.,
\begin{equation}\label{013}
\dim\mathfrak{c}_\mathfrak{g}(u)=\dim\mathfrak{c}
(\mathfrak{c}_\mathfrak{g}(u))=
\mathrm{rk}\mathfrak{g}.
\end{equation}
On the other hand,
(1) of Lemma \ref{lemma-6} indicates
\begin{equation}\label{014}
\dim\mathfrak{c}_\mathfrak{g}(u)
\geq\dim\mathfrak{g}-2\dim\mathfrak{h}=\dim\mathfrak{h}
+\mathrm{rk}\mathfrak{g}.
\end{equation}
Compare (\ref{013}) and (\ref{014}), we get $\dim\mathfrak{h}=0$ and $\mathrm{rk}\mathfrak{g}=
\dim\mathfrak{g}$, i.e., $\mathfrak{g}$ is Abelian.
This is a contradiction.
\end{proof}

Secondly, we need the classification list in \cite{Wo1968}, for a nonsymmetric strongly isotropy irreducible $G/H$ on which a compact connected $G$
acts almost effectively. We list their Lie algebra pairs in Table
\ref{table-1}.

\begin{table}
  \centering
  \begin{tabular}{|c|c|c|c|c|c|}
    \hline
   $\mathfrak{g}$ & $\dim\mathfrak{g}$ & $\mathfrak{h}$&  $\dim\mathfrak{h}$ & Conditions\\ \hline \hline
    $su(pq)$ & $p^2q^2-1$ &$su(p)\oplus su(q)$
       & $p^2+q^2-2$& $p\geq q\geq 2,pq>4$ \\ \hline
    $su(16)$ & 255 & $so(10)$ & 45 & \\ \hline
     $su(27)$ & 728 &$E_6$ & 78& \\ \hline
     $su(\tfrac{n(n-1)}{2})$ &$\tfrac{n^2(n-1)^2}{4}-1 $ &
       $su({n})$ &$n^2-1$ &$n\geq5$\\ \hline
    $su(\tfrac{n(n+1)}{2})$ & $\tfrac{n^2(n+1)^2}{4}-1 $
      & $su(n)$ & $n^2-1$ & $n\geq 3$\\ \hline
     $sp(2)$ & 10 & $so(3)$ & 3 & \\ \hline
     $sp(7)$ & 105 & $sp(3)$ & 21 & \\ \hline
     $sp(10)$ & 210 & $su(6)$ & 35 &`\\ \hline
     $sp(16)$ & 528 & $so(12)$ & 66 & \\ \hline
    $sp(28)$ & 1596&$E_7$ & 133 & \\ \hline
     $so(20)$ & 190 & $su(4)$ & 15 & \\ \hline
     $so(70)$ & 2415 & $su(8)$ & 63 & \\ \hline
     $so(n^2-1)$ & $\tfrac{(n^2-1)(n^2-2)}{2}$ & $su(n)$ & $n^2-1$ &$n\geq 3$\\ \hline
     $so(16)$ & 120 &$so(9)$ & 36 &\\ \hline
     $so(2n^2+n)$ &$\tfrac{(2n^2+n)(2n^2+n-1)}{2}$ &
    $so(2n+1)$ & $n(2n+1)$ & $n\geq 2$\\ \hline
     $so(2n^2+3n)$ &
    $\tfrac{(2n^2+3n)(2n^2+3n-1)} {2}$&$so(2n+1)$ &$n(2n+1)$& $n\geq2$\\ \hline
     $so(42)$ &861 & $sp(4)$ & 36 &\\ \hline
     $so(2n^2-n-1)$ & $\tfrac{(2n^2-n-1)(2n^2-n-2)}{2}$
    &$sp(n)$ & $2n^2+n$ & $n\geq3$\\ \hline
     $so(2n^2+n)$ & $\tfrac{(2n^2+n)(2n^2+n-1)}{2}$
    & $sp(n)$ & $2n^2+n$ & $n\geq3$\\ \hline
     $so(128)$ & 8128 & $so(16)$ & 120 &\\ \hline
     $so(2n^2-n)$ & $\tfrac{(2n^2-n)(2n^2-n-1)}{2}$ & $so(2n)$ & $n(2n-1)$ & $n\geq 4$\\ \hline
     $so(2n^2+n-1)$ & $\frac{(2n^2+n-1)(2n^2+n-2)}{2}$& $so(2n)$& $n(2n-1)$ &$n\geq4$\\ \hline
     $so(7)$ & 21 & $G_2$ & 14 &\\ \hline
     $so(14)$ & 91 & $G_2$ & 14 &\\ \hline
     $so(26)$ & 325 & $F_4$ & 52 &\\ \hline
     $so(52)$ & 1326 & $F_4$ & 52 &\\ \hline
     $so(78)$ & 3003 & $E_6$ & 78 &\\ \hline
     $so(133)$ & 8778 & $E_7$ & 133 &\\ \hline
     $so(248)$ & 30628 & $E_8$ & 248 &\\ \hline
     $G_2$ & 14 & $so(3)$ & 3 &\\ \hline
     $G_2$ & 14 & $su(3)$ & 8 &\\ \hline
     $F_4$ & 52 & $so(3)\oplus G_2$ & 17 & \\ \hline
     $F_4$ & 52 & $su(3)\oplus su(3)$ & 18 &\\ \hline
     $E_6$ & 78 & $su(3)$ & 8 &\\ \hline
     $E_6$ & 78 & $G_2$ & 14 &\\ \hline
     $E_6$ & 78 & $ su(3)\oplus G_2$ & 22 &\\ \hline
     $E_6$ & 78 & $su(3)\oplus su(3)\oplus su(3)$
    & 24 &\\ \hline
     $E_7$ & 133 & $su(3)$ & 8 &\\ \hline
     $E_7$ & 133 & $sp(3)\oplus G_2$ & 35 &\\ \hline
     $E_7$ & 133 & $su(2)\oplus F_4$ & 55 &\\ \hline
     $E_7$ & 133 & $su(3)\oplus su(6)$ & 43 &\\ \hline
     $E_8$ & 248 & $G_2\oplus F_4$ & 66 & \\ \hline
     $E_8$ & 248 & $su(9)$ & 80 &\\ \hline
     $E_8$ & 248 & $su(3)\oplus E_6$ & 86 & \\ \hline
     $sp(n)$ & $2n^2+n$ & $sp(1)\oplus so(n)$ &
    $\tfrac{n(n-1)}{2}+3$ & $n\geq 3$\\ \hline
     $so(4n)$ & $2n(4n-1)$ & $sp(1)\oplus sp(n)$  & $2n^2+n+3$ & $n\geq 3$\\
    \hline
  \end{tabular}
  \caption{Lie algebra pairs for compact non-symmetric strongly isotropy irreducible spaces}\label{table-1}
\end{table}

\begin{remark}
The pair $(so(4n), sp(1)\oplus sp(n))$ in Table \ref{table-1} is a symmetric pair when $n=2$. Using a graph automorphism for $so(8)$, we can change it to the standard symmetric pair $(so(8), so(3)\oplus so(5))$ for a real Grassmannian.
\end{remark}

\begin{proof}[Proof of Theorem \ref{main-thm-4}]
Let $G/H$ be a nonsymmetric strongly isotropy irreducible
homogeneous manifold on which the compact connected semi simple $G$ acts almost effectively. Then $(\mathfrak{g},\mathfrak{h})$ is listed in
Table \ref{table-1}.

We check each case in Table \ref{table-1} and see that (3) of Lemma \ref{lemma-6} is only satisfied when $(\mathfrak{g},\mathfrak{h})=(so(7),G_2)$, $(G_2,su(3))$
or in Table \ref{table-2}.

\begin{table}
  \centering
  \begin{tabular}{|c|c|c|c|c|}
    \hline
   No. & $\mathfrak{g}$ & $\dim\mathfrak{g}$&  $\mathfrak{h}$&  $\dim\mathfrak{h}$ \\ \hline
    1 & $su(6)$ &  35 & $su(2)\oplus su(3)$  & 11 \\ \hline
    2 & $sp(2)$ & 10 & $so(3)$& 3 \\ \hline
    3 & $F_4$ & 52 & $su(2)\oplus G_2$ & 17 \\ \hline
    4 & $F_4$ & 52 & $su(3)\oplus su(3)$ & 18 \\ \hline
    5 & $E_7$ & 133 & $so(3)\oplus F_4$ & 55 \\ \hline
    6 & $E_7$ & 133 & $su(3)\oplus su(6)$ & 43 \\ \hline
    7 & $E_8$ & 248 & $su(3)\oplus E_6$ & 86 \\ \hline
    8 & $so(12)$ & 66 & $sp(1)\oplus sp(3)$ & 24 \\ \hline
    9 & $so(16)$ & 120 & $sp(1)\oplus sp(4)$ & 39 \\ \hline
    \hline
  \end{tabular}
  \caption{Lie algebra pairs in Table \ref{table-1} which  satisfy (3) of Lemma \ref{lemma-6}}\label{table-2}
\end{table}

So, to prove Theorem \ref{main-thm-4}, we only need to
assume conversely that $G/H$ is each one in Table \ref{table-2}, and check case by case for contradictions.
In the upcoming case-by-case discussion,  we apply the following conventions in \cite{XW2016}. We choose a Cartan subalgebra $\mathfrak{t}$ of $\mathfrak{g}$ such that $\mathfrak{t}\cap\mathfrak{h}$ is a Cartan subalgebra of $\mathfrak{h}$. Using $\langle\cdot,\cdot\rangle_{\mathrm{bi}}$, the root systems $\Delta_\mathfrak{g}$ for $ \mathfrak{g} $ and $\Delta_\mathfrak{h}$ for $ \mathfrak{h} $
are viewed as subsets in $\mathfrak{t}$ and $\mathfrak{t}\cap\mathfrak{h}$ respectively.
The inner product $\langle\cdot,\cdot\rangle_{\mathrm{bi}}$ and the orthonormal basis  $\{e_1,\cdots,e_n\}$ are suitably chosen such that $\Delta_\mathfrak{g}$  and $\Delta_\mathfrak{h}$
can be canonically presented. All roots and all root planes are with respect to $\mathfrak{t}$ or its intersection with the specified subalgebras.

{\bf Case 1}: $(\mathfrak{g},\mathfrak{h})=
(su(6),su(2)\oplus su(3))$.

In this case, the root system $\Delta_\mathfrak{g}=\{\pm(e_i-e_j),\forall 1\leq i<j\leq 6\}$, and $\mathfrak{t}\cap\mathfrak{h}$ is spanned by
$e_1+e_3+e_5-e_2-e_4-e_6$ from the $su(2)$-summand, and
$e_1+e_2-e_3-e_4$, $e_3+e_4-e_5-e_6$ from the $su(3)$-summand. So $\mathfrak{t}\cap\mathfrak{m}$ consists of all vectors of the form $ae_1-ae_2+be_3-be_4+ce_5-ce_6$,  $\forall a,b,c\in\mathbb{R}$ with $a+b+c=0$.
A generic vector in $\mathfrak{t}\cap\mathfrak{m}$, for example, $u=e_1-e_2+2e_3-2e_4-3e_5+3e_6$,
is a regular vector in $\mathfrak{g}$. Then $\dim\mathfrak{c}_\mathfrak{g}(u)=5$ does not satisfy (1) of Lemma \ref{lemma-6}, which is a contradiction.

{\bf Case 2}: $(\mathfrak{g},\mathfrak{h})=(sp(2),so(3))$.

This $G/H$ is in fact the Berger space $Sp(2)/SU(2)$ in the classification for positively curved homogeneous manifolds
\cite{Be1961,DX2017}.
We have the root system $\Delta_\mathfrak{g}=\{\pm e_1,\pm e_2,\pm e_1\pm e_2\}$, and $\mathfrak{t}\cap\mathfrak{m}$ is spanned by
$u=(e_1+e_2)-(-e_1)=2e_1+e_2$. This $u$ is a regular vector in $\mathfrak{g}$. So $\dim\mathfrak{c}_\mathfrak{g}(u)=2$ does not satisfy (1) of Lemma \ref{lemma-6}, which is a contradiction.

{\bf Case 3 or 4}: $(\mathfrak{g},\mathfrak{h})
=(F_4, su(2)\oplus G_2)$ or $(F_4, su(3)\oplus su(3))$.

By (1) of Lemma \ref{lemma-6}, $\dim\mathfrak{c}_\mathfrak{g}(u)\geq 18$ or $16$. By the classification for flag manifolds (see \cite{Al1997} or Section \ref{section-2-5}, same below), we see $\dim\mathfrak{c}(\mathfrak{c}_\mathfrak{g}(u))=1$ for any $u\in\mathfrak{m}\backslash\{0\}$, with $\mathrm{Ad}(G)u=
F_4/Spin(7)U(1)$ or $\mathrm{Ad}(G)u=F_4/Sp(3)U(1)$.
Then by (2) of Lemma \ref{lemma-6},
$$31=\dim\mathrm{Ad}(G)u+\dim\mathfrak{c}
(\mathfrak{c}_\mathfrak{g}(u))\geq\dim\mathfrak{m}=35 \mbox{ or } 34$$
provides a contradiction.

{\bf Case 5}: $(\mathfrak{g},\mathfrak{h})=(E_7, F_4\oplus su(2))$.

Let $u'$ be a generic vector in $\mathfrak{t}\cap\mathfrak{m}$. Then $\mathfrak{c}_{\mathfrak{g}}(u')=so(8)\oplus\mathbb{R}^3$.
This observation needs more explanation, which is put in the appendix at the end of Section 5.

We can find a vector $u\in\mathfrak{m}$ such that
$u$ is sufficiently close to $u'$ and $\mathrm{Ad}(H)u$
is a principal orbit. Then $u$ is contained in a conic open dense subset $\mathcal{U}$ indicated by Lemma \ref{lemma-3}. The centralizer $\mathfrak{c}_{\mathfrak{g}}(u)$ must be isomorphic to
$\mathfrak{c}_\mathfrak{g}(u')$, because otherwise by the classification for flag manifolds, we have $\dim\mathfrak{c}(\mathfrak{c}_\mathfrak{g}(u))\geq4$ and
$\dim\mathfrak{c}_\mathfrak{g}(u)\leq 19$, which contradicts
(1) of Lemma \ref{lemma-6}.

As indicated by the table in Theorem 11.1 of \cite{Wo1968},
the $\mathrm{Ad}(H)$-action on $\mathfrak{m}$ is the tensor product between the natural $SO(3)$-action on $\mathbb{R}^3$ and the isotropy representation for the symmetric space $E_6/F_4$. So $\mathrm{Ad}(H)$-action on $\mathfrak{m}$ is faithful. On the other hand, it is not in Table B of \cite{HH1970}, i.e., any principal $\mathrm{Ad}(H)$-orbit in $\mathfrak{m}$ has the same dimension as $H$. So we have $\mathfrak{c}_\mathfrak{h}(u)=0$.

Now we consider the linear map $l(\cdot)=[u,\cdot]_\mathfrak{h}|_{\mathfrak{m}}$ from $\mathfrak{m}$ to $\mathfrak{h}$, which appears in the proof of Lemma \ref{lemma-6}. For any $w\in\mathfrak{m}$, we have
$\langle[u,w],\mathfrak{h}\rangle_{\mathrm{bi}}
=\langle w,[\mathfrak{h},u]\rangle_{\mathrm{bi}}$, so $\ker l$ is the $\langle\cdot,\cdot\rangle_{\mathrm{bi}}$-orthogonal
complement of $[\mathfrak{h},u]$ in $\mathfrak{m}$. Because $\mathfrak{c}_\mathfrak{h}(u)=0$, $\dim\ker l=\dim\mathfrak{m}-\dim\mathfrak{h}=23$.

Finally, (\ref{035}) can be applied to this $u\in\mathcal{U}$, which provides $\dim\mathfrak{c}_\mathfrak{g}(u)=23$. This contradicts the previous observation $\mathfrak{c}_\mathfrak{g}(u)=so(8)\oplus\mathbb{R}^3$.

{\bf Case 6}: $(\mathfrak{g},\mathfrak{h})=(E_7, su(3)\oplus su(6))$.

By (1) of Lemma \ref{lemma-6}, we have
\begin{equation}\label{015}
\dim\mathfrak{c}_{\mathfrak{g}}(u)\geq 47\quad\mbox{and}
\quad\dim \mathrm{Ad}(G)u\leq 86
\end{equation}
for any $u\in\mathfrak{m}$. By the classification for flag manifolds, if $\dim\mathfrak{c}(\mathfrak{c}_\mathfrak{g}(u))\geq 3$,
$\dim\mathfrak{c}_{\mathfrak{g}}(u)\leq 31$ with the equality achieved when $\mathrm{Ad}(G)u=E_6/Spin(8)T^3$. So we
have $\max_{v\in\mathfrak{m}} \dim\mathfrak{c}(\mathfrak{c}_\mathfrak{g}(v))\leq 2$, and
by (2) of Lemma \ref{lemma-6} and the second inequality in (\ref{015}), for a generic $u\in\mathfrak{m}$,
$$88\geq\dim\mathrm{Ad}(G)u+\max_{v\in\mathfrak{m}} \dim\mathfrak{c}(\mathfrak{c}_\mathfrak{g}(v))\geq \mathfrak{m}=90.$$
This is a contradiction.

{\bf Case 7}: $(\mathfrak{g},\mathfrak{h})=(E_8, su(3)\oplus E_6)$.

On one hand,
by (1) of Lemma \ref{lemma-6}, we have
$\dim\mathfrak{c}_\mathfrak{g}(u)\geq 76$ for each $u\in\mathfrak{m}$. By the classification of flag manifolds,
we must have
\begin{equation}\label{034}
\dim\mathfrak{c}(\mathfrak{c}_\mathfrak{g}(u))\leq 2,\quad\forall u\in\mathfrak{m},
\end{equation}
 otherwise
$\dim\mathfrak{c}_\mathfrak{g}(u)\leq 48$ with the equality achieved when $\mathrm{Ad}(G)u=E_8/Spin(10)T^3$.

On the other hand, there is a $\mathfrak{k}=E_7$ in $\mathfrak{g}$ which contains the $E_6$-summand $\mathfrak{h}_1$ in $\mathfrak{h}$ and intersects the $su(3)$-summand $\mathfrak{h}_2$ in $\mathfrak{h}$
at a line. The pair $(\mathfrak{k},
\mathfrak{k}\cap\mathfrak{h}))
=(E_7,E_6\oplus\mathbb{R})$ is a symmetric pair. Denote by $\mathfrak{m}'$ the $\langle\cdot,\cdot\rangle_{\mathrm{bi}}$-complement of $\mathfrak{k}\cap\mathfrak{h}$
in $\mathfrak{k}$, then $\mathfrak{k}=(\mathfrak{k}\cap\mathfrak{h})+\mathfrak{m}'$
is a Cartan decomposition. The rank of $E_7/E_6 U(1)$ is 3 \cite{He1962}, i.e., we have find a 3-dimensional commutative subspace $\mathfrak{t}'$ in $\mathfrak{m}'$, from which we can find a vector $u$ with $\dim\mathfrak{c}(\mathfrak{c}_\mathfrak{g}(u))\geq3$.
Since $\mathrm{rk}\mathfrak{g}=\mathrm{rk}\mathfrak{h}$ and
$\mathfrak{k}=E_7$ is regular in $\mathfrak{g}=E_8$
 (i.e., each root plane of $\mathfrak{k}$ is a root plane
of $ \mathfrak{g} $),
$\mathfrak{m}'$ and $\mathfrak{m}$ are both sums of root planes of $ \mathfrak{g} $, so we have $\mathfrak{m}'\subset\mathfrak{m}$. The previously mentioned
$u\in\mathfrak{m}'\subset\mathfrak{m}$ satisfies
with $\dim\mathfrak{c}(\mathfrak{c}_\mathfrak{g}(u))\geq3$.
This contradicts (\ref{034}).

{\bf Case 8}: $(\mathfrak{g},\mathfrak{h})=(so(12),sp(1)\oplus sp(3))$.

In this case, the root system $\Delta_\mathfrak{g}=\{\pm e_i\pm e_j,\forall 1\leq i<j\leq 6\}$, and the subspace $\mathfrak{t}\cap\mathfrak{h}$ is linearly spanned by
$e_1+\cdots+e_6$ from the $sp(1)$-summand and $e_1-e_2$, $e_3-e_4$, $e_5-e_6$ from the $sp(3)$-summand. So $\mathfrak{t}\cap\mathfrak{m}$ consists of all vectors of the form $ae_1+ae_2+be_3+be_4+ce_5+ce_6$, $\forall a,b,c\in\mathbb{R}$ with $a+b+c=0$.

For the vector $u=e_1+e_2+2e_3+2e_4-3e_5-3e_6\in
\mathfrak{t}\cap\mathfrak{m}$, the dimension of $\mathfrak{c}_\mathfrak{g}(u)=su(2)\oplus su(2)\oplus su(2)\oplus\mathbb{R}^3$ is $12$. This  contradicts (1) of Lemma \ref{lemma-6}.

{\bf Case 9}: $(\mathfrak{g},\mathfrak{h})=(so(16),sp(1)\oplus sp(4))$.

In this case,
the root system $\Delta_\mathfrak{g}=\{\pm e_i\pm e_j,\forall 1\leq i<j\leq 8\}$, and the subspace $\mathfrak{t}\cap\mathfrak{h}$ is linearly spanned by
$e_1+\cdots+e_8$ from the $sp(1)$-summand and $e_1-e_2$, $e_3-e_4$, $e_5-e_6$, $e_7-e_8$ from the $sp(4)$-summand. So $\mathfrak{t}\cap\mathfrak{m}$ consists all vectors of the form $ae_1+ae_2+be_3+be_4+ce_5+ce_6+de_7+de_8$,  $\forall a,b,c,d\in\mathbb{R}$ with $a+b+c+d=0$.

For the vector $u=e_1+e_2+2e_3+2e_4+3e_5+3e_6-6e_7-6e_8\in
\mathfrak{t}\cap\mathfrak{m}$, the dimension of $\mathfrak{c}_\mathfrak{g}(u)=su(2)\oplus su(2)\oplus su(2)\oplus su(2)\oplus \mathbb{R}^4$ is $16$. This  contradicts (1) of Lemma \ref{lemma-6}.

This ends the proof of Theorem \ref{main-thm-4}.
\end{proof}
\bigskip

\noindent
{\bf Appendix: some discussion for the algebraic structure of $(\mathfrak{g},\mathfrak{h})=(E_7,su(2)\oplus F_4)$}


For simplicity, we denote $\mathfrak{h}_1$ and $\mathfrak{h}_2$ the $su(2)$- and $F_4$-summands in
$\mathfrak{h}$ respectively.

Notice that $\mathfrak{t}\cap\mathfrak{h}_1$ is a line, which
centralizer in $\mathfrak{g}=E_7$ has a semi simple summand
$\mathfrak{g}'=E_6$. The subalgebra $\mathfrak{h}_2=F_4$ is contained in $\mathfrak{g}'=E_6$ such that $E_6/F_4$ is a symmetric space. We can expand a Cartan subalgebra of $\mathfrak{h}_2=F_4$ first to $\mathfrak{g}'$ and then to $\mathfrak{g}$, which provides a Cartan subalgebra $\mathfrak{t}$ such that $\mathfrak{t}\cap\mathfrak{h}_1$,
$\mathfrak{t}\cap\mathfrak{h}_2$, $\mathfrak{t}\cap\mathfrak{g}'$ are Cartan subalgebras for
$\mathfrak{h}_1$, $\mathfrak{h}_2$ and $\mathfrak{g}'$ respectively. The roots and root planes, with respect to these
specified Cartan subalgebras can be arranged as following.
Using $\langle\cdot,\cdot\rangle_{\mathrm{bi}}$, roots are viewed as vectors in $\mathfrak{t}$ rather than $\mathfrak{t}^*$.

The root system $\Delta_{\mathfrak{g}}$ of $\mathfrak{g}$ consists of the following roots:
\begin{eqnarray*}
& &\pm e_i\pm e_j,\forall 1\leq i<j\leq 6; \pm\sqrt{2}e_7;\\
& &\pm\tfrac{1}{2}e_1\pm\cdots
\pm\tfrac{1}{2}e_6\pm\tfrac{\sqrt{2}}{2}e_7
\mbox{ with even }+\tfrac12 \mbox{-coefficients}.
\end{eqnarray*}
It has a prime root system
\begin{eqnarray*}
& &\alpha_1=e_1-e_2,\ \alpha_2=e_2-e_3,\ \alpha_3=e_3-e_4,
\ \alpha_4=e_4-e_5,\\
& &\alpha_5=e_5-e_6,\ \alpha_6=e_5+e_6,\ \alpha_7=-\tfrac12(e_1+\cdots+e_6+\sqrt{2}e_7).
\end{eqnarray*}
The subset $\{\alpha_2,\cdots,\alpha_7\}$ is the prime root system of $\mathfrak{g}'=E_6$.

The subalgebra $\mathfrak{h}_2=F_4$ is the fixed point set for
the involutive automorphism $\sigma$ of $\mathfrak{g}'$
which maps each $\alpha_i$ to $\alpha_{9-i}$ for $i=2,3,6,7$, and fixes $\alpha_4$ and $\alpha_5$. So $\mathfrak{t}\cap\mathfrak{h}_2$ is spanned by $\alpha_4$, $\alpha_5$, $\alpha_3+\alpha_6=e_3-e_4+e_5+e_6$ and
$\alpha_2+\alpha_7=\tfrac12(-e_1+e_2-3e_3-e_4-e_5-e_6-\sqrt{2}e_7)$.

The subspace $\mathfrak{t}\cap\mathfrak{h}_1$ in $\mathfrak{h}_1=su(2)$ commutes with each root plane of $\mathfrak{h}_2=F_4$, so it commutes with each root plane of $\mathfrak{g}'=E_6$. Then we see that $\mathfrak{t}\cap\mathfrak{h}_1$ is
$\langle\cdot,\cdot\rangle_{\mathrm{bi}}$-orthogonal to
$\mathfrak{t}\cap\mathfrak{g}'$, i.e., it is spanned by $2e_1-\sqrt{2}e_7$.

Above description is enough for us to calculate  $\mathfrak{t}\cap\mathfrak{m}$, which is linearly spanned by $e_3-e_4-e_5-e_6$ and $e_1+3e_2+\sqrt{2}e_7$.

Let $u'$ be a generic vector in $\mathfrak{t}\cap\mathfrak{m}$. For example, we can choose $u'=7(e_1+3e_2+\sqrt{2}e_7)+5(e_3-e_4-e_5-e_6)$. The
centralizer $\mathfrak{c}_{\mathfrak{g}}(u')$ has the following roots,
\begin{eqnarray*}
& &\pm(e_3+e_4),\ \pm(e_3+e_5),\ \pm(e_3+e_6),\ \pm(e_4-e_5),\ \pm(e_4-e_6), \ \pm(e_5-e_6),\nonumber\\
& &\pm\tfrac12(e_1-e_2+\sqrt{2}e_7)\pm
\tfrac12(e_3-e_4+e_5+e_6),\nonumber\\
& &\pm\tfrac12(e_1-e_2+\sqrt{2}e_7)
\pm\tfrac12(e_3+e_4-e_5+e_6),\nonumber\\
& &\pm\tfrac12(e_1-e_2+\sqrt{2}e_7)
\pm\tfrac12(e_3+e_4+e_5-e_6),
\end{eqnarray*}
which provides a root system of $so(8)$. So $\mathfrak{c}_{\mathfrak{g}}(u')=so(8)\oplus\mathbb{R}^3$.

\begin{remark} Another description for $(E_7,su(2)\oplus F_4)$ can be found in Table 35 of \cite{Dy1952}.
\end{remark}

\section{Homogeneous manifold on which all invariant metrics are Berwald}

In this section, we  prove Theorem C, which classifies homogeneous manifold $G/H$ with a compact semi simple $G$, on which all $G$-invariant metrics on Berwald. It is an immediate corollary of Theorem B and the following.

\begin{theorem} \label{thm-1}
Let $G/H$ be a homogeneous manifold with a compact $(G,\langle\cdot,\cdot\rangle_{\mathrm{bi}})$.
Then $G/H$ is a Finsler equigeodesic space if and only if each $G$-invariant Finsler metric on $G/H$ is Berwald.
\end{theorem}

\begin{proof}Firstly, we assume $G/H$ is Finsler equigeodesic and prove each $G$-invariant Finsler metric $F$ on $G/H$ is Berwald.
By Lemma \ref{proposition-1}, $F$ is naturally reductive
with respect to the orthogonal reductive decomposition $\mathfrak{g}=\mathfrak{h}+\mathfrak{m}$.
It has a vanishing spray vector field, so By Lemma \ref{lemma-9}, it is Berwald. This proves one side of theorem \ref{thm-1}.

Nextly, we assume each $G$-invariant Finsler metric on $G/H$ is Berwald and prove $G/H$ is Finsler equigeodesic. Denote $S=\{u| u\in\mathfrak{m}, |u|_{\mathrm{bi}}=\langle u,u\rangle^{1/2}_{\mathrm{bi}}=1 \}$ the unit sphere
in $\mathfrak{m}$. By Lemma \ref{lemma-8}, we just need to prove that each $u\in S$ is a Finsler equigeodesic vector.
We may assume $\mathrm{Ad}(H)u\neq S$, otherwise $G/H$ is a compact rank-one Riemannian symmetric space \cite{Wa1952}, which is obviously Finsler equigeodesic.

Let $F$ be any $G$-invariant Finsler metric on $G/H$. We use the same $F=F(o,\cdot)$ to denote the $\mathrm{Ad}(H)$-invariant Minkowski norm on $\mathfrak{m}$. We can find two sufficiently small $\mathrm{Ad}(H)$-invariant open neighborhood $\mathcal{U}_1$ and $\mathcal{U}_2$ of the orbit $\mathrm{Ad}(H)u$ in $S$ with $\mathcal{U}_1\subset\mathcal{U}_2$, and an $\mathrm{Ad}(H)$-invariant smooth cut-off function $f:S\rightarrow[0,1]$ satisfying $f(\mathcal{U}_1)=1$ and $f(S\backslash\mathcal{U}_2)=0$. Notice that $S\backslash\mathcal{U}_2$ contains a nonempty open subset of $S$.
%
For $t\in\mathbb{R}$ which is sufficiently close to $0$, $$F_t(y)=\sqrt{\langle y,y\rangle_{\mathrm{bi}}+t f(\tfrac{y}{|y|_{\mathrm{bi}}}) F(y)^2}$$
defines a smooth family of $\mathrm{Ad}(H)$-invariant Minkowski norms on $\mathfrak{m}$. We use the same $F_t$ to denote the corresponding $G$-invariant Finsler metrircs  on $G/H$.

Denote
$\langle\cdot,\cdot\rangle_y$ and $\langle\cdot,\cdot\rangle_y^{F_t}$ the fundamental tensors of the Minkowski norms $F$ and $F_t$
respectively. Let $\eta_t:\mathfrak{m}\backslash\{0\}\rightarrow\mathfrak{m}$ be
the spray vector field of $F_t$ and $\omega=\tfrac{d}{dt}|_{t=0}\eta_t$. Then we have the following observations.
Notice that $F_0=|\cdot|_{\mathrm{bi}}$, so $\langle\cdot,\cdot\rangle_{y}^{F_0}=\langle\cdot,\cdot\rangle_{\mathrm{bi}}$ and $\eta_0=0$.
For $y\in\mathfrak{m}\backslash\mathbb{R}_{\geq0}\mathcal{U}_2$, $\langle\cdot,\cdot\rangle^{F_t}_y=\langle\cdot,\cdot\rangle_{\mathrm{bi}}$, and for
$y\in\mathbb{R}_{>0}\mathcal{U}_1$, $\tfrac{d}{dt}\langle\cdot,\cdot\rangle_y^{F_t}=\langle\cdot,\cdot\rangle_y$. Since each $F_t$ is Berwald, By Lemma \ref{lemma-9},
each $\eta_t$ is quadratic, and then $\omega$ is also quadratic.

The definition of spray vector field provides
\begin{equation}\label{052}
\langle\eta_t(y),w\rangle^{F_t}_y=\langle y,[w,y]_\mathfrak{m}\rangle^{F_t}_y,\quad\forall y\in\mathfrak{m}\backslash\{0\},w\in\mathfrak{m}.
\end{equation}
By the observations in the previous paragraph, the derivative of (\ref{052}) for the $t$-variable
which is evaluated at $t=0$ can be presented as
\begin{equation}\label{053}
\langle \omega(y),w\rangle_{\mathrm{bi}}=\tfrac{d}{dt}|_{t=0}\langle y,[w,y]_\mathfrak{m}\rangle^{F_t}_y,\quad\forall
y\in\mathfrak{m}\backslash\{0\},w\in\mathfrak{m}.
\end{equation}
The left side of (\ref{053}) is quadratic for $y$, and the right side vanishes for $y\in\mathfrak{m}\backslash\mathbb{R}_{\geq0}\mathcal{U}_2$.
So both sides of (\ref{053}) vanish for all $y\in\mathfrak{m}\backslash\{0\}$.
In particular, for $y=u\in\mathbb{R}_{>0}\mathcal{U}_1$, we have
$$\tfrac{d}{dt}|_{t=0}\langle u,[w,u]_\mathfrak{m}\rangle_u=\langle u,[w,u]_\mathfrak{m}\rangle_u=0,\quad\forall u\in\mathfrak{m}.$$

To summarize, $u\in S\subset\mathfrak{m}\backslash\{0\}$ is a geodesic vector for any $G$-invariant Finsler metric $F$, so $u$ is a Finsler equigeodesic vector. Since $u$ is arbitrary, $G/H$ is Finsler equigeodesic.
This proves the other side of Theorem \ref{thm-1}.
\end{proof}

\noindent{\bf Acknowledge}\quad
This paper is supported by National Natural Science Foundation of China (12131012, 12001007, 11821101),
Beijing Natural Science Foundation (1222003),
Natural Science Foundation of Anhui province (1908085QA03). The authors sincerely thank Valerii Beres- tovskii,  Huibin Chen, Zhiqi Chen, Shaoqiang Deng, Yurii Nikonorov, Zaili Yan, Fuhai Zhu and Wolfgang Ziller for helpful discussions.

\end{document}